\documentclass[a4paper,12pt]{amsart}

\usepackage{amsmath}
\usepackage{amsthm}
\usepackage{amssymb}
\usepackage{amsfonts}
\usepackage{enumerate}
\usepackage{mathrsfs}
\usepackage[shortlabels]{enumitem} 
\usepackage[dvipsnames]{xcolor}
\usepackage[normalem]{ulem}
\usepackage{marginnote}
\usepackage{esint}
\usepackage{hyperref}
\usepackage{pdfpages}
\usepackage{multicol}

\usepackage[alphabetic]{amsrefs}

\calclayout

\theoremstyle{plain}
\newtheorem{theorem}{Theorem}[section]

\newtheorem{lemma}[theorem]{Lemma}
\newtheorem{corollary}[theorem]{Corollary}
\newtheorem{proposition}[theorem]{Proposition}

\newtheorem*{claim*}{Claim}

\newtheorem{maintheorem}{Theorem}

\newtheorem{maincorollary}[maintheorem]{Corollary}

\theoremstyle{definition}
\newtheorem*{definition*}{Definition}
\newtheorem{definition}[theorem]{Definition}
\newtheorem{example}[theorem]{Example}
\newtheorem{question}{Question}

\theoremstyle{remark}
\newtheorem{remark}[theorem]{Remark}

\numberwithin{equation}{section}

\def \B {\mathcal{B}}

\def \R {\mathbb{R}}

\def \E {\mathbb{E}}
\def \H {\mathcal{H}}
\def \M {\mathcal{M}}
\def \N {\mathbb{N}}
\def \P {\mathbb{P}}
\def \U {\mathcal{U}}

\def \Z {\mathbb{Z}}

\def \eps {\varepsilon}

\def \sbs {\subseteq}

\DeclareMathOperator{\dist}{dist}
\DeclareMathOperator{\diam}{diam}

\DeclareMathOperator{\LF}{LF}
\DeclareMathOperator{\TV}{TV}
\DeclareMathOperator{\TC}{TC}

\DeclareMathOperator{\Prob}{Prob}
\DeclareMathOperator{\PART}{PART}

\DeclareMathOperator{\Var}{Var}

\begin{document}
\title{$L_1$ Actions and Embeddings of Property A Spaces}

\begin{abstract}
We provide several new characterizations of Property A for bounded degree graphs. In particular, we show that $(X,d)$ has Property A if and only if there is a proper gauge $\omega$ such that the Lipschitz free space $\operatorname{LF}(X,\omega\circ d)$ is isomorphic to $\ell_1$. As a consequence, all finitely generated groups with Property A admit proper uniformly Lipschitz affine actions on $\ell_1$. Moreover, for groups with finite Nagata dimension, we obtain actions with compression exponent 1. This result applies to higher rank lattices, such as $\operatorname{SL}(3,\mathbb{Z})$. We also show that a countable discrete group coarsely embeds into $L_1$ if and only if it admits a proper uniformly Lipschitz affine action on a subspace of $L_1$.
\end{abstract}


\author{Chris Gartland}
\address{Department of Mathematics and Statistics, University of North Carolina at Charlotte}

\author{Ignacio Vergara}
\address{Departamento de Matemática y Estadística, Universidad de La Frontera}

\author{Tianyi Zheng}
\address{Department of Mathematics, University of California, San Diego}

\thanks{The first named author was supported by the National Science Foundation under Grant Number DMS-2546184. The second named author was supported by the ANID-SIA project 85250079. The third named author was supported by the National Science Foundation under Grant Number DMS-2348143}
	


\date{June 2026}

\maketitle

\tableofcontents

\section{Introduction} \label{sec:intro}
We say that a connected, bounded degree graph $X$ has {\it Property A} if for every $\eps > 0$, there exist $S<\infty$ and a mapping $\Theta: X \to \Prob(X)$ into probability measures on $X$ such that, for all $x\in X$, the measure $\Theta(x)$ is supported in the ball $B_S(x)$ (where $X$ is equipped with the unweighted shortest path metric), and for all edges $x\sim y\in X$, we have the total variation distance bound $\|\Theta(x)-\Theta(y)\|_{\TV} < \eps$. Property A is a coarse invariant introduced by Yu in \cite{Yu}. He proved that groups with Property A coarsely embed into $L_1$ and that groups coarsely embedding into $L_1$ satisfy the coarse Baum--Connes conjecture\footnote{Yu's results were actually stated for Hilbert space $L_2$ instead of $L_1$, but since $L_2$ and $L_1$ are mutually coarsely embeddable into each other (see \S\ref{ss:Banach}), it is equivalent to use $L_1$.}. The definition of Property A given here is not Yu's original definition but is (by now) well-known to be equivalent to it. We refer the reader to \cite{Willett} for further information on, and many other equivalent characterizations of, this fundamental coarse invariant.

It was an open question for some time whether the converse to Yu's theorem is true -- that is, whether the following question had a positive answer.
\begin{question} \label{q:Yuconverse}
If a connected, bounded degree graph coarsely embeds into $L_1$, must it have Property A?
\end{question}
\noindent It turns out that Question~\ref{q:Yuconverse} has a negative answer. Arzhantseva, Guentner, and Špakula constructed in \cite{Z2homology} a counterexample using a sequence of $\Z_2$-homology covers of graphs. Their example is actually defined as a disjoint union of finite graphs, but it can be modified in order to obtain an infinite connected graph satisfying the same properties. Later, Osajda \cite[\S1.2]{Osajda} constructed a finitely generated group without Property A but with the Haagerup property (a stronger, equivariant version of coarse embeddability into $L_1$). In this article, we are able to recover a complete characterization of Property A in terms of $L_1$-embeddings, in the spirit of Question~\ref{q:Yuconverse}. The idea is to ask for a stronger type of embedding into $L_1$ out of our bounded degree graph. This is facilitated by bringing into the fold the Wasserstein-1 and Lipschitz free spaces. The Wasserstein-1 space $W_1(X)$ is (the completion of) the set of finitely supported probability measures on $X$ equipped with a certain optimal transport metric, see \cite[Part~I, \S6]{Villani} for the precise definition and background. The Lipschitz free space $\LF(X)$ is (the completion of) the vector space of finitely supported 0 total mass signed measures on $X$ equipped with the Kantorovich dual norm $\|\mu\|_{\LF} = \sup\{|\int_Xf d\mu| : f: X\to\R$ is 1-Lipschitz$\}$. The most salient facts to record here about these spaces are that the Dirac delta map $X \ni x \mapsto \delta_x \in W_1(X)$ is an isometric embedding, and, for any fixed reference measure $\mu_0 \in W_1(X)$, the map $W_1(X) \ni \mu \mapsto \mu-\mu_0 \in \LF(X)$ is an isometric embedding. Thus, an embedding of $W_1(X)$ or $\LF(X)$ into $L_1$ is stronger than an embedding of $X$ into $L_1$. Recall that a proper gauge is a concave homeomorphism $\omega: [0,\infty) \to [0,\infty)$. The following theorem is our main result, yielding a complete embedding-theoretic characterization of Property A.

\begin{maintheorem}[Embedding-Theoretic Characterization of Property A] \label{thm:main}
Let $X$ be a connected, bounded degree graph and $d$ the shortest path metric on $X$. Then the following are equivalent:
\begin{enumerate}
    \item $X$ has Property A.
    \item There exists a proper gauge $\omega$ such that $\LF(X,\omega\circ d)$ is isomorphic to $\ell_1$.
    \item There exists a metric space $Y$ coarsely containing $X$ such that $W_1(Y)$ biLipschitzly embeds into $L_1$.
\end{enumerate}
\end{maintheorem}
In the next subsection, \S\ref{ss:extended}, we state and prove an extended version of Theorem~\ref{thm:main} (Theorem~\ref{thm:mainextended}) that gives many other equivalent properties to the three above, including stochastic embeddings into ultrametric\footnote{Recall that an {\it ultrametric} $d$ on a set is a metric satisfying the stronger triangle inequality $d(x,z) \leq \max\{d(x,y),d(y,z)\}$.} spaces, Poincar\'e inequalities for point processes, and containment of ``subquadratic isoperimetry small-scale expanders". We defer to \S\ref{ss:extended} for further discussion of these terms. Some crucial ingredients in our proof are Elek's alternate characterizations of Property A in terms of hyperfiniteness properties \cite{Elek} (see Theorem~\ref{thm:Elek}). Theorem~\ref{thm:main} follows from Theorem~\ref{thm:mainextended}.

A second motivation for Theorem~\ref{thm:main} comes from group actions on Banach spaces. Every amenable group admits a proper isometric action on $L_p$ for every $p\in[1,\infty)$; see e.g. \cite[Theorem 42]{Now15}. For $1 \leq p \leq 2$, the existence of such an action is known as the Haagerup property or a-T-menability, which has had a tremendous impact on analytic group theory and related areas; see \cite{CCJJV} for details. Since Property A for a group is often thought of as a non-equivariant version of amenability, it is natural to ask what kinds of actions can be obtained from this property. However, as far as we know, this question has not yet been explored in this level of generality. All known results so far assume some additional geometric structure in order to construct actions with a certain desired behavior. Perhaps the most celebrated result in this direction is Yu's construction \cite{Yu2} of proper isometric actions on $\ell_p$ for every hyperbolic group, where $p$ is a sufficiently large number depending on the group; see also \cite{Nic}.

In this paper, we are interested in constructing proper actions on $\ell_1$. Since we want to encompass all groups with Property A, we cannot hope to obtain a general result in terms of isometric actions since this class includes many examples of groups with Property (T), for which all affine isometric actions on a subspace of $L_1$ have a bounded orbit; see \cite[Theorem 1.3]{BaFuGeMo}. However, Theorem \ref{thm:main} has a strong consequence in terms of uniformly Lipschitz actions.

Recall that a uniformly Lipschitz action\footnote{A group action on a metric space is {\it uniformly Lipschitz} if there exists $L<\infty$ such the map $x\mapsto\gamma\cdot x$ is $L$-Lipschitz for every group element $\gamma$.} on a metric space $(X,d_X)$ of a countable discrete group $\Gamma$ equipped with a proper\footnote{A metric on a discrete space is {\it proper} if all its bounded subsets are finite.}, left-invariant metric $d_\Gamma$ is {\it proper} if for some point $x_0\in X$, we have $d_X(x_0,\gamma \cdot x_0) \to \infty$ as $d_\Gamma(e,\gamma)\to\infty$. An immediate consequence of Theorem~\ref{thm:main} is the following result on group actions on $\ell_1$. 

\begin{maintheorem}[Proper Uniformly Lipschitz Affine Action on $\ell_1$]
Let $\Gamma$ be a finitely generated group\footnote{We always equip a finitely generated group $\Gamma$ with a left-invariant word metric $d_\Gamma$ obtained by identifying the group with its Cayley graph with respect to some finite, symmetric generating set.} with Property A. Then $\Gamma$ admits a proper, uniformly Lipschitz affine action on $\ell_1$.
\end{maintheorem}

\begin{proof}
Since the assignment $X \mapsto \LF(X)$ is functorial, whenever $\Gamma$ acts properly and isometrically on $X$, it acts properly by affine isometries on $\LF(X)$ (see \cite[Lemma~6.1]{Ver2} for details). Hence, in this case, if $\LF(X)$ is isomorphic to $\ell_1$, we get a proper, uniformly Lipschitz affine action of $\Gamma$ on $\ell_1$. Since for any proper gauge $\omega$, $\Gamma$ acts properly by isometries on $(\Gamma,\omega\circ d_\Gamma)$, the conclusion follows from Theorem~\ref{thm:main}. 
\end{proof}

At this point, it is natural to ask which groups admit proper, uniformly Lipschitz affine actions on $\ell_1$, or more generally on a subspace of $L_1$. The only obvious obstruction is that the group admits no coarse embedding into $L_1$. Our next theorem shows that this is in fact the only obstruction.

\begin{maintheorem}[Embeddability Implies Equivariant Embeddability] \label{thm:embedding=>equivariantembedding}
Let $\Gamma$ be a countable discrete group equipped with a proper, left-invariant metric. Then $\Gamma$ admits a proper, uniformly Lipschitz affine action on a subspace of $L_1$ if and only if $\Gamma$ coarsely embeds into $L_1$.
\end{maintheorem}

One of the main tools we use to prove $(1) \implies (2)$ in Theorem~\ref{thm:main} is a ``coarse-to-Lipschitz dictionary" (Proposition~\ref{prop:dictionary}), by which we use proper gauges to obtain a correspondence between several common notions in coarse and Lipschitz geometry. The proof of Theorem~\ref{thm:embedding=>equivariantembedding} uses this dictionary and a result of Vergara \cite[Theorem 1.1]{Ver} characterizing uniformly Lipschitz affine actions on subspaces of $L_1$ via almost invariant conditionally negative definite kernels (see Theorem~\ref{thm:AICNDK}). We also point out that this result allows one to construct actions with Lipschitz constants arbitrarily close to 1. The proof of Theorem~\ref{thm:embedding=>equivariantembedding} can be found in \S\ref{ss:embedding=>equivariantembedding}.

If $\Gamma$ is a finitely generated group with finite Nagata dimension (see \S\ref{sss:dimension}), a condition stronger than Property A, we can obtain a more quantitative form of properness. Let $(X,d_X)$ be a metric space. Let us say that $\Gamma$ admits uniformly Lipschitz actions on $X$ with {\it compression exponent 1} if for some point $x_0\in X$ and for all $\eps > 0$, there exists a uniformly Lipschitz action of $\Gamma$ on $X$ and $C,K<\infty$ with $d_X(x_0,\gamma \cdot x_0) \geq C^{-1}d_\Gamma(e,\gamma)^{1-\eps}-K$ for all $\gamma\in\Gamma$. If we can find such an action for $\eps=0$, we say that the action has {\it undistorted} or {\it quasi-isometrically embedded} orbit. There are various constructions of uniformly Lipschitz actions with undistorted orbits available in the literature when the group $\Gamma$ exhibits some tree-like geometry (e.g., \cite[\S6]{DrutuMackay}, \cite[Theorem~E]{Gartland}, \cite{Ver2}). It turns out that in the general setting of groups with finite Nagata dimension, the tree-like geometry can be obtained by postcomposing the word metric with a gauge of the form $\omega(t)=t^{1-\eps}$. This tree-like geometry follows from results in \cite{Gartland}, and we are able to obtain the following theorem -- see Corollary~\ref{cor:nearlyundistortedorbit} for the proof.

\begin{maintheorem}[Nearly Undistorted Orbit] \label{thm:nearlyundistortedorbit}
Let $\Gamma$ be a finitely generated group. If $\Gamma$ has finite Nagata dimension, then $\Gamma$ admits uniformly Lipschitz affine actions on $\ell_1$ with compression exponent 1.
\end{maintheorem}

This result has the following notable consequence for higher rank lattices such as $\operatorname{SL}(n,\Z)$ ($n\geq 3$).

\begin{maincorollary}\label{cor:higherranklattice}
Let $G$ be a connected, semisimple Lie group of real rank at least 2, and let $\Gamma$ be an irreducible lattice in $G$. Then $\Gamma$ admits uniformly Lipschitz affine actions on $\ell_1$ with compression exponent 1.
\end{maincorollary}
\begin{proof}
By \cite[Corollary 7.13]{HiPe}, the Nagata dimension of $G$ is exactly $\operatorname{dim}(G)$. Moreover, by \cite[Theorem~A]{LMR}, $\Gamma$ biLipschitzly embeds into $G$. In particular, it has finite Nagata dimension. Thus, $\Gamma$ admits uniformly Lipschitz affine actions on $\ell_1$ with compression exponent 1 by Theorem~\ref{thm:nearlyundistortedorbit}.
\end{proof}

Let $\Gamma$ be a higher rank lattice as in Corollary \ref{cor:higherranklattice}. By Oppenheim \cite{Opp} and de Laat--de la Salle \cite{dLdlS}, $\Gamma$ always has a fixed point when acting by uniformly Lipschitz affine maps on a uniformly convex Banach space, such as $L_p$ for $1<p<\infty$. This should be contrasted with the situation for \emph{isometric} affine actions, where a finitely generated group acts properly on $L_1$ if and only if it acts properly on $L_p$ for $p \in [1,2]$ (see \cite[Theorem 41]{Now15}). In a way, this says that when relaxing the requirement of an affine action from being isometric to being uniformly Lipschitz, $L_p$ retains some rigidity for $p \in (1,\infty)$, while $L_1$ becomes much more flexible.

\subsection{Extended version of Theorem~\ref{thm:main}, proof structure, and outline} \label{ss:extended}

In this subsection, we state the aforementioned extended version of Theorem~\ref{thm:main} and sketch explanations for where the theorems yielding the various implications can be found in the article. In order to understand the theorem statement, the reader might want to review a number of definitions, which we point to next.

Proper gauges are discussed in \S\ref{ss:gauges}.

The definitions of separation modulus and stochastic embeddings into ultrametric spaces can be found in \S\ref{sss:separation} and Definition~\ref{def:stochasticembedding}. They are important notions first developed in theoretical computer science \cite{Bartal} to find approximation algorithms -- see \cite[\S4]{FRT} for some examples.

The Lipschitz free space $\LF(X)$ and Wasserstein-1 space $W_1(X)$ are defined in \S\ref{ss:LF}. The free space $\LF(X)$ is a canonical Banach space into which $X$ isometrically embeds. In the simple case of $X = [0,1]$, it is folklore that $\LF([0,1]) = L_1([0,1])$, and a natural and old problem has been to determine for which $X$ is $\LF(X)$ is isomorphic to or isomorphically embeddable into $L_1$ (for example, \cite[Remark~2]{Pestov}). See \cite{Gartland}, \cite{GO}, and references therein for prior results on this problem and discussion on its importance in theoretical computer science. We believe that Theorem~\ref{thm:mainextended} provides a significant step towards our understanding of this problem.

The Bernoulli$(p)$ nonempty point process $\widetilde{\Pi}(p)$ and its Poincar\'e constant $C(\widetilde{\Pi}(p),H)$ are defined in Definitions~\ref{def:Bernoullinonempty} and \ref{def:Poincare}. In a nutshell, the Bernoulli$(p)$ point process $\Pi(p)$ on a finite graph $H$ is formed by including points into the set $\Pi(p)$ independently and with probability $p$, and $\widetilde{\Pi}(p)$ is obtained from $\Pi(p)$ by conditioning on the event $\Pi(p) \neq \emptyset$. The Poincar\'e constant $C(\widetilde{\Pi}(p),H)$ is defined to be the least constant $C<\infty$ such that
\begin{equation} \label{eq:Poincareintro}
    \E\left[\left|{\textstyle \int_H f du_{\widetilde{\Pi}(p)} -\int_H f du_{H}} \right|\right]\le\frac{C}{|H|}\sum_{x\sim y\in H}\left|f(x)-f(y)\right|
\end{equation}
for all $f:H \to \R$, where $u_X$ denotes the uniform probability measure on a nonempty subset $X\sbs H$, and the sum is over all edges $\{x,y\}\sbs H$, $x\sim y$.

Subquadratic isoperimetry small-scale expander sequences are defined in Definition~\ref{def:almostexpander}. They are defined to be a bounded degree sequence of finite graphs $(H_k)_k$ with linear isoperimetry $|A| \lesssim |\partial A|$ for sets $A\sbs H_k$ up to cardinality $k$ (hence the ``small-scale expander" part of the name), and satisfying $|A| \lesssim |\partial A|^{\beta}$ for some $\beta < 2$ for sets $A\sbs H_k$ of cardinality between $k$ and $\frac{1}{2}|H_k|$ (hence the ``subquadratic isoperimetry" part of the name). This requirement of linear control on the lower part of the isoperimetric profile is precisely {\it non-uniform local amenability}, which was shown to be equivalent to the failure of Property A by Elek \cite{Elek}. The subquadratic control for the larger scale part of the isoperimetric profile is chosen in order to make the implication $(6)\implies(7)$ below hold.

Finally, we may state our extended characterization of Property A.

\begin{theorem}[Extended Characterization of Property A] \label{thm:mainextended}
Let $X$ be an infinite, connected, bounded degree graph and $d$ the shortest path metric on $X$. Then the following are equivalent:
\begin{enumerate}
    \item $X$ has Property A.
    \item There exists a proper gauge $\omega$ such that $(X,\omega\circ d)$ has finite separation modulus.
    \item There exists a proper gauge $\omega$ such that $(X,\omega\circ d)$ admits an $L_1$-stochastic embedding into the class of ultrametric spaces.
    \item There exists a proper gauge $\omega$ such that $\LF(X,\omega\circ d)$ is isomorphic to $\ell_1$.
    \item There exists a metric space $Y$ coarsely containing $X$ such that $W_1(Y)$ biLipschitzly embeds into $L_1$.
    \item $X$ does not contain a sequence of induced subgraphs $(H_k)_k$ such that $|H_k|\to\infty$ and $\sup_k C(\widetilde{\Pi}(p_k),H_k) < \infty$ for some sequence of probabilities $(p_k)_k$ with $p_k\to 0$.
    \item $X$ does not contain a subquadratic isoperimetry small-scale expander sequence of induced subgraphs.
\end{enumerate}
\end{theorem}

\begin{proof}~\\
\indent The implication $(1) \implies (2)$ is Theorem~\ref{thm:PropA=>separation}.

The implication $(2) \implies (3)$ is Theorem~\ref{thm:separation=>stochasticembedding}.

The implication $(3) \implies (4)$ is Theorem~\ref{thm:stochasticembedding=>LF}.

The implication $(4) \implies (5)$ follows immediately from the facts that (i) the identity map $(X,d) \to (X,\omega\circ d)$ is a coarse equivalence for any proper gauge $\omega$, (ii) $W_1(Z)$ isometrically embeds into $\LF(Z)$ for any metric space $Z$, and (iii) $\ell_1$ isometrically embeds into $L_1$.

The implication $(5) \implies (6)$ is Theorem~\ref{thm:Poincare=>L1distortion}.

The implication $(6) \implies (7)$ is Theorem~\ref{thm:almostexpander=>Poincare}.

The implication $(7) \implies (1)$ is Theorem~\ref{thm:nonPropA=>almostexpander}.
\end{proof}

\begin{remark} \label{rmk:inbetween}
Any property sitting between any two of the properties $(1)$-$(7)$ above is, of course, equivalent to each property above. For example, the property that there exists a metric space $Y$ coarsely containing $X$ such that $\LF(Y)$ isomorphically embeds into $L_1$ is implied by $(4)$ and implies $(5)$, and is thus equivalent to each of $(1)$-$(7)$.
\end{remark}

\begin{remark}
Theorem~\ref{thm:mainextended} allows us to obtain a new characterization of amenability for finitely generated residually finite groups. Let $\Gamma$ be a residually finite group with finite symmetric generating set $S$. Let $(\Gamma_n)_n$ be a nested sequence of finite index normal subgroups of $\Gamma$ with $\cap_n \Gamma_n = \{1\}$, and equip each group $\Gamma/\Gamma_n$ with the word metric $d_n$ with respect to $S/\Gamma_n$.  Then we obtain the characterization: {\it $\Gamma$ is amenable if and only if there is a proper gauge $\omega$ and constant $C<\infty$ such that $\LF(\Gamma/\Gamma_n,\omega\circ d_n)$ admits a $C$-isomorphic embedding into $L_1$ for every $n$.} Indeed, we first record that $\Gamma$ being amenable is equivalent to a coarse disjoint union $(\square \Gamma,d)$ of $(\Gamma/\Gamma_nd_n)_n$ having Property A (see \cite[Proposition 11.39]{Roe}). By Remark~\ref{rmk:inbetween}, this latter property is equivalent to $\LF(\square \Gamma,\omega \circ d)$ admitting an isomorphic embedding into $L_1$ for some proper gauge $\omega$. It is an elementary fact that $\LF(\square \Gamma,\omega \circ d)$ admits an isomorphic embedding into $L_1$ if and only if there exists $C<\infty$ such that $\LF(\Gamma/\Gamma_n,\omega\circ d_n)$ admits a $C$-isomorphic embedding into $L_1$ for every $n$. This shows the desired equivalence.
\end{remark}

One may compare the equivalence $(1)\iff(6)$ in Theorem~\ref{thm:mainextended} with characterizations of non-coarse embeddability into $L_1$ via Poincar\'e inequalities \cite{Tessera}, \cite[Theorem~7.6]{Ostrovskiibook}.

\subsubsection{Outline}
The outline of the article is as follows:

In \S\ref{sec:prelims}, we review all the preliminary definitions, concepts, and results we need for the article. While the results of this section are generally not new, we often need to synthesize existing results in the literature or slightly modify them, which at times requires a short argument.

In \S\ref{sec:rescaling}, we prove the Rescaling Lemma~\ref{lem:rescaling}, which has the Coarse-to-Lipschitz Dictionary (Proposition~\ref{prop:dictionary}) as a consequence. This dictionary is one of the main ingredients used in both the proof of $(1)\implies(2)$ of Theorem~\ref{thm:mainextended} and in the proof of Theorem~\ref{thm:embedding=>equivariantembedding}.

In \S\ref{sec:equivariantembeddability}, we use the Coarse-to-Lipschitz Dictionary (Proposition~\ref{prop:dictionary}) along with the main result of \cite{Ver} to prove Theorem~\ref{thm:embedding=>equivariantembedding}.

In \S\ref{sec:PropA}, we prove the implications $(1)\implies(2)\implies(3)$ of Theorem~\ref{thm:mainextended}. The first of these implications follows from an application of the Coarse-to-Lipschitz Dictionary (Proposition~\ref{prop:dictionary}) and, importantly, Elek's hyperfiniteness characterization of Property A \cite{Elek} (see Theorem~\ref{thm:Elek}). The second implication follows from a quantitative estimate on the $L_1$-stochastic distortion into ultrametric spaces of snowflakes of finite separation modulus spaces (Theorem~\ref{thm:separation=>stochasticembedding}), a result we believe to be of independent interest.

In \S\ref{sec:nonPropA}, we prove the implications $(5)\implies(6)\implies(7)\implies(1)$ of Theorem~\ref{thm:mainextended}. Elek's hyperfiniteness characterization of Property A again plays an essential role, this time in the proof of $(7)\implies(1)$. The proof of $(5)\implies(6)$ employs a new method of using Poincar\'e inequalities of the form \eqref{eq:Poincareintro} in order to obtain lower bounds on the $L_1$-distortion of Wasserstein-1 spaces.

\section{Preliminaries} \label{sec:prelims}

\subsection{Banach spaces} \label{ss:Banach}
We recall some standard facts about Banach and $L_p$ spaces. The reader may consult \cite{JL,AO} for reference.

Let $L<\infty$. We say that a seminormed\footnote{A {\it seminorm} is a function $\|\cdot\|$ on a vector satisfying all the axioms of a norm except possibly $\|x\|=0 \implies x = 0$.} space $E_1$ {\it $L$-isomorphically embeds} into a normed space $E_2$ if there exists a noncontractive linear map from $E_1$ into $E_2$ with operator norm at most $L$. If the operator is also surjective, we say that $E_1$ is {\it $L$-isomorphic} to $E_2$. When $L=1$, we write {\it isometric} instead of 1-isomorphic. If $E_2$ is a Banach space and $E_1 \sbs E_2$ is a closed subspace, then we say that $E_1$ is {\it $L$-complemented} in $E_2$ if there exists a linear projection $E_2\to E_1$ of norm $\leq L$. In this case, $E_1$ is isomorphic to $E_2 \oplus W$, where $W$ denotes the kernel of the projection.

Let $p\in[1,\infty]$. We write $L_p$ to denote the Lebesgue space $L_p(\R)$ and $\ell_p$ the sequence space. Whenever $\mu$ is a separable purely nonatomic measure, $L_p(\mu)$ is linearly isometric to $L_p$, and so we abuse notation and write $L_p$ for $L_p(\mu)$ in this case. If $\mu$ is purely atomic with countably many atoms, then $L_p(\mu)$ is linearly isometric to $\ell_p$, and so we again abuse notation and write $\ell_p$ for $L_p(\mu)$ in this case. For every measure $\mu$, every separable subspace of $L_p(\mu)$ isometrically embeds into $L_p$. For $m\in\N$, we write $\ell_p^m$ to denote $\R^m$ equipped with the $p$-norm. The spaces $\ell_p$ and $\ell_p^m$ are always isometric to a complemented subspaces of $L_p$. It holds that $L_2$ linearly isometrically embeds into $L_p$ for $p \leq 2$ (e.g., \cite[Corollary~8.8]{BL}), and the metric space $(L_p,\|\cdot\|_p^{p/2})$ isometrically embeds into $L_2$ (e.g., \cite[Remark~5.10]{MNquotients}).

\subsection{Coarse and Lipschitz geometry}
\subsubsection{Embeddings}
Let $(X,d_X)$ and $(Y,d_Y)$ be semimetric\footnote{We define a {\it semimetric} $d$ on a set $Z$ to be a function of two variables satisfying all the axioms of a metric except possibly for $d(x,y) = 0 \implies x=y$.} spaces. A map $f:X\to Y$ is a \emph{coarse embedding} if there are nondecreasing functions $\rho_1,\rho_2:[0,\infty)\to[0,\infty)$ such that $\rho_1(t)\to\infty$ as $t\to\infty$ and
\begin{align*}\label{CE}
	\rho_1(d_X(x,y)) \leq d_Y(f(x),f(y)) \leq \rho_2(d_X(x,y))
\end{align*}
for all $x,y\in X$. If such a map $f$ exists, we say that $Y$ {\it coarsely contains} $X$. The functions $\rho_1$ and $\rho_2$ are referred to as the {\it lower and upper control functions}, respectively. If there are $C<\infty$ and $s\in (0,\infty)$ such that $\rho_1(t) \geq st$ and $\rho_2(t) \leq Cst$ for all $t\geq 0$, then we say that $f$ is a {\it $C$-biLipschitz embedding}, or just {\it biLipschitz embedding} if we don't want to keep track of the constant.

\subsubsection{Dimension} \label{sss:dimension}
Let $n\in\N$. We say that a metric space $X$ has {\it asymptotic dimension $\leq n$}, if there exists a function $R: [0,\infty)\to[0,\infty)$ (which can be taken to be nondecreasing), called the {\it $n$-dimensional control function}, such that for all scales $r<\infty$, there exists a covering $\B$ of $X$ such that
\begin{itemize}
    \item $\diam(B) \leq R(r)$ for every $B\in\B$ and
    \item for every $A \sbs X$ with $\diam(A) < r$, the cardinality of the set $\{B\in\B: B \cap A \neq \emptyset\}$ is at most $n+1$.
\end{itemize}
If the control function can be taken to satisfy $R(r) \leq \gamma r$ for some $\gamma<\infty$, then $X$ is said to have {\it (Assouad)-Nagata dimension $\leq n$}, and the smallest such $\gamma$ is called the {\it Nagata $n$-dimensional modulus} of $X$. The least such $n$ for which $X$ has asymptotic dimension $\leq n$ is called the {\it asymptotic dimension} of $X$, and the least such $n$ for which $X$ has Nagata dimension $\leq n$ is called the {\it Nagata dimension} of $X$. Nagata dimension was introduced in the seminal paper \cite{LS}.

\subsubsection{Separation} \label{sss:separation}
Let $(X,d)$ be a metric space. Given a partition $P$ of $X$, we write $P(x)$ to denote the unique element of $P$ containing $x$. For $\Delta < \infty$, we say that a partition $P$ of $X$ is {\it $\Delta$-bounded} if $\diam(P(x)) \leq \Delta$ for every $x\in X$. By associating a partition with the equivalence relation it generates, we view the set of $\Delta$-bounded partitions of $X$ as a compact subspace of $\{0,1\}^{X\times X}$ in the natural way.

We say that $X$ has \emph{controlled separation} if there is a function $\Delta: (0,\infty) \to [0,\infty)$, called the \emph{separation control function}, such that for every $\eps>0$, there exists a Radon probability measure $\P$ on the set of $\Delta(\eps)$-bounded partitions of $X$ such that, for every $x,y \in M$, $\P_P(P(x)\neq P(y)) \leq \eps d(x,y)$. By a simple compactness argument, it is easy to see that a function $\Delta$ is a separation control function for $X$ if and only if it is a separation control function for every finite subset of $X$. If $\Delta$ can be chosen to satisfy $\Delta(\eps) \leq \frac{\sigma}{\eps}$ for some $\sigma < \infty$, then $X$ is said to have \emph{finite separation modulus}, and the least such $\sigma$ is called the \emph{separation modulus} of $X$. Note that if $X$ is 1-separated, meaning $d(x,y) \geq 1$ for all $x\neq y\in X$, then the inequality is trivially satisfied for $\eps > 1$, and so we can take the domain of $\Delta$ to be $(0,1]$. Also note that if a separation control function exists, it can be taken to be nonincreasing. Separation modulii first appeared in \cite{Bartal}, see \cite[$\S$1.7.3]{Naor} for further discussion and history of this fundamental biLipschitz invariant.

\subsection{Proper gauges} \label{ss:gauges}
Following \cite{Kalton}, we say that a function $\omega: [0,\infty) \to [0,\infty)$ is a {\it gauge} if $\omega$ is concave and $\lim_{t\to0^+}\omega(t)=\omega(0)=0$. Gauges are necessarily continuous, and they satisfy a number of useful inequalities such as
\begin{itemize}
    \item $s \leq t \implies \omega(s) \leq \omega(t)$ (nondecreasing)
    \item $\omega(s+t) \leq \omega(s)+\omega(t)$ (subadditive)
    \item $s\leq 1 \implies \omega(st) \geq s\omega(t)$
    \item $s\geq 1 \implies \omega(st) \leq s\omega(t)$
\end{itemize}
We will be using these inequalities throughout the article. The first two imply that whenever $d$ is a metric on a set $X$ and $\omega$ is a gauge, $\omega \circ d$ is also a metric on $X$ (inducing the same topology as $d$).

We say that a gauge $\omega$ is {\it proper} if $\lim_{t\to\infty}\omega(t) = \infty$. Equivalently, $\omega: [0,\infty)\to[0,\infty)$ is a proper gauge if and only if $\omega$ is a concave homeomorphism. If $(X,d)$ is a metric space $d$ and $\omega$ is a proper gauge, then the identity map between $(X,d)$ and $(X,\omega\circ d)$ is a coarse equivalence (because the map and its inverse are both coarse embeddings). Note that the composition of two (or more) proper gauges is another proper gauge. Examples of proper gauges that we frequently use are the snowflake gauges $t\mapsto t^\alpha$ for $\alpha\in (0,1]$.

\subsection{Property A}
In this subsection, we specialize to the case when $X$ is a connected, bounded degree graph, and the metric $d$ is the shortest path metric on $X$ (in this article, shortest path metrics are always with respect to unit edge lengths). We will employ the following notation and convention throughout the article: a {\it graph} is a set $X$ together with a symmetric relation $\sim$. The {\it edges} of $X$ are the doubleton subsets $\{x,y\}$ with $x\sim y$. The {\it degree} of a vertex $x \in X$ is $|\{y\in X: x\sim y\}|$, and $X$ has {\it bounded degree} if there exists $r\in\N$ such that $x$ has degree at most $r$ for every $x\in X$. The graph $X$ is {\it connected} if the relation $\sim$ generates an equivalence relation with only one equivalence class. The shortest path metric on connected graphs is defined as usual and results in a 1-separated metric space. An {\it induced subgraph} of $X$ is a subset $H\sbs X$ equipped with the restriction of $\sim$ to $H$. A subset $H$ is {\it connected} if the subgraph it induces is connected.

Recall the definition of $X$ having Property A from \S\ref{sec:intro}. When $X$ is the Cayley graph of a group $\Gamma$ with respect to a finite symmetric generating set, we say that $\Gamma$ has Property A if the graph $X$ has Property A. Since any two such Cayley graphs are biLipschitz equivalent and since Property A is inherited coarse embeddings (in particular, under biLipschitz equivalence), this property is well-defined independent of the choice of generating set. Likewise, we say that $\Gamma$ has finite Nagata dimension or $\Gamma$ coarsely embeds into a metric space $Y$ if $X$ has finite Nagata dimension or $X$ coarsely embeds into a metric space $Y$. A connected, bounded degree graph having finite Nagata dimension implies that it has Property A, which in turn implies that it is coarsely embeddable into $L_1$. See \cite{Willett} for these facts (Corollary~2.2.11 and Theorem~1.1.7) and for further background on Property A. There are numerous equivalent characterizations of Property A in the literature. We focus on two that are most relevant for our purposes: local hyperfiniteness and local strong hyperfiniteness. We will generally follow \cite{Elek} in defining these terms. Although the precise definitions we give are not verbatim the same as those given in \cite{Elek}, they are readily seen to be equivalent, see Remark~\ref{rmk:equivalent}.

\begin{definition}[Hyperfiniteness] \label{def:hyperfinite}
Let $\eps>0$, $k\in\N$, and $H$ a finite graph. Let us say that $H$ is {\it $(\eps,k)$-hyperfinite} if there exist subsets $\{A_i\}_i,Y$ of $H$ such that
\begin{itemize}
    \item $\{A_i\}_i$ partitions $H$,
    \item $|A_i| \leq k$ for all $i$,
    \item $|Y| < \eps |H|$, and
    \item for every connected subset $C$ of $H \setminus Y$, there is an $i$ such that $C \sbs A_i$.
\end{itemize}
We say that a bounded degree graph $X$ is {\it locally hyperfinite} if for every $\eps>0$, there exists $k\in\N$ such that every finite induced subgraph of $X$ is $(\eps,k)$-hyperfinite.
\end{definition}

\begin{remark} \label{rmk:equivalent}
Elek uses in \cite{Elek} the following definition of a finite graph $H$ being $(\eps,k)$-hyperfinite: there exists $Y \sbs H$ such that $|Y| < \eps|H|$ and every connected subset of $H\setminus Y$ has cardinality $\leq k$. It is easy to see that this definition of $(\eps,k)$-hyperfinite is implied by the one given in Definition~\ref{def:hyperfinite}.

Conversely, suppose that $H$ is $(\eps,k)$-hyperfinite in the sense used in \cite{Elek} and that every vertex of $H$ has degree at most $r\in\N$. Let $Y \sbs H$ be a subset of minimal cardinality such that every connected subset of $H\setminus Y$ has cardinality $\leq k$. Then it must hold that for every $y \in Y$, there exists $x \in H\setminus Y$ with $x \sim y$. Indeed, we would otherwise have the existence of a $y \in Y$ with $x \not\sim y$ for every $x \in H\setminus Y$, and then the set $Y' := Y \setminus \{y\}$ would have strictly smaller cardinality than $Y$ and would satisfy that every connected subset of $H \setminus Y'$ has cardinality $\leq k$. Let $\{A'_i\}_i$ be the connected components of $H\setminus Y$ (which have cardinality $\leq k$ by assumption), assign to each $y\in Y$ a vertex $x_y \in H\setminus Y$ with $x\sim y$, and set $A_i := \cup_{x_y \in A_i'} \{y\} \cup A_i'$. Then $\{A_i\}_i$ partitions $H$, $|A_i| \leq rk$ for all $i$, $|Y| < \eps|H|$, and every connected subset of $H\setminus Y$ is contained in $A_i$ for some $i$. This shows that $H$ is $(\eps,rk)$-hyperfinite in the sense of Definition~\ref{def:hyperfinite}.

We can conclude that a bounded degree graph is locally hyperfinite in the sense of Definition \ref{def:hyperfinite} if and only if it is locally hyperfinite in the sense used in \cite{Elek}.
\end{remark}

A subset $Y \sbs H$ is called a \emph{$k$-separator} if every connected subset of $H \setminus Y$ has cardinality at most $k$. We say that $H$ is {\it strongly $(\eps,k)$-hyperfinite} if there exists a probability measure $\P$ on the set of $k$-separators of $H$ such that, for every $x \in H$, $\P_Y(x\in Y) \leq \eps$. We say that a bounded degree graph $X$ is {\it locally strongly hyperfinite} if for every $\eps>0$, there exists $k\in\N$ such that every finite induced subgraph of $X$ is strongly $(\eps,k)$-hyperfinite.

The following theorem is the main result of \cite{Elek}.

\begin{theorem}[Hyperfiniteness Characterization of Property A] \label{thm:Elek}
A bounded degree graph has Property A $\iff$ it is locally hyperfinite $\iff$ it is locally strongly hyperfinite.
\end{theorem}
\noindent An exposition and self-contained proofs for other equivalent characterizations of Property A can be found in Elek--Tim\'ar~\cite{Elek-Timar}.

There is a clear relationship between controlled separation and local strong hyperfiniteness. Let $H$ be a finite induced subgraph of a connected  graph $X$. We let $d_X$ denote the shortest path metric on $X$ and equip $H$ with the metric $d_X$. There is a canonical map $\Phi$ from the set of $k$-separators of $H$ to the set of $(k-1)$-bounded partitions of $(H,d_X)$. Indeed, suppose that $Y \sbs H$ is a $k$-separator. Then we define an equivalence relation $\Phi(Y): H \times H \to \{0,1\}$ by $\Phi(Y)(x,y) = 1 \iff$ $x=y$ or $\{x,y\}$ belongs to a connected subset of $H \setminus Y$. Hence, if $x$ is equivalent to $y$ via $\Phi(Y)$, then there exists a path with at most $k-1$ edges connecting $x$ to $y$, which says that the diameter (with respect to $d_X$) of the equivalence class of $x$ is at most $k-1$. This results in the following lemma.

\begin{lemma} \label{lem:controlledseparation}
Let $X$ be a connected, bounded degree graph and $d_X$ the shortest path metric on $X$. If $X$ is locally strongly hyperfinite, then $(X,d_X)$ has controlled separation.
\end{lemma}

\begin{proof}
Assume that $X$ is locally strongly hyperfinite. It suffices to find a function $\Delta: (0,1]\to[0,\infty)$ (as previously mentioned, we may take the domain of $\Delta$ to be $(0,1]$ here since $X$ is 1-separated) such that $\Delta$ is a separation control function on $(H,d_X)$ for every finite subset of $H \sbs X$. Since $X$ is locally strongly hyperfinite, for every $\eps>0$, there exists $k_\eps\in\N$ such that every finite induced subgraph of $X$ is strongly $(\eps/2,k_\eps)$-hyperfinite. We will see that the function $\Delta(\eps) := k_\eps - 1$ works. 

Let $H \sbs X$ be finite. For every $x,y\in H$, choose a geodesic $\gamma_{xy} = \{x = x_0, x_1, x_2, \dots x_n = y\} \sbs X$ with $|\gamma_{xy}| = d_X(x,y)+1$. Let $H' = \cup_{x,y\in H}\gamma_{xy}$. Let $\P$ be a probability measure on the set of $k$-separators of the induced subgraph $H'$ such that $\P_Y(z\in Y) \leq \eps/2$ for every $z \in H'$. Let $\Phi$ be the map from the set of $k_\eps$-separators of $H'$ to the set of $(k_\eps-1)$-bounded partitions of $H'$, as described prior to the statement of the lemma, so that $\Phi_\#\P$ is a probability measure on the set of $\Delta(\eps)$-bounded partitions of $(H',d_X)$. Restricting each partition to $H$ yields a probability measure, that we continue to denote by $\Phi_\#\P$, on the set of $\Delta(\eps)$-bounded partitions of $(H,d_X)$.

Now fix $x\neq y \in H$. Clearly, if $\Phi(Y)(x,y) = 0$, then there exists $z \in \gamma_{xy}$ such that $z \in Y$. Therefore, we have that
\begin{align*}
    (\Phi_\#\P)_P(P(x) \neq P(y)) \leq \P_Y(\cup_{z\in\gamma_{xy}} \{z \in Y\}) \leq |\gamma_{xy}|\frac{\eps}{2} = \frac{\eps}{2}(d(x,y)+1) \leq \eps d(x,y).
\end{align*}
\end{proof}

\subsection{Lipschitz free and Wasserstein-1 spaces} \label{ss:LF}
Let $(X,d)$ be a metric space and $\M_{00}(X)$ the set of finitely supported signed measures on $X$ with vanishing total mass. We define a norm on $\M_{00}(X)$ by $\|\mu\|_{\LF} := \sup\{|\int_X fd\mu|: f:X\to\R$ is 1-Lipschitz$\}$. We call the completion of the resulting normed space the {\it Lipschitz free space} of $X$, denoted by $\LF(X)$. When $\mu,\nu$ are two finitely supported probability measures on $X$, their difference $\mu-\nu$ belongs to $\LF(X)$, and the distance $\|\mu-\nu\|_{\TC}$ is called the {\it Wasserstein-1} distance\footnote{Other common names include {\it Kantorovich metric} and {\it earth mover's distance}.} between $\mu$ and $\nu$, denoted $W_1(\mu,\nu)$. The completion of probability measures under this metric is called the {\it Wasserstein-1 space}, denoted $W_1(X)$. When we wish to emphasize the metric, we will write $\LF(X,d)$ and $W_1(X,d)$. An important fact we shall use is that whenever $Y$ is a metric space and $X \sbs Y$, the natural inclusion $W_1(X) \sbs W_1(Y)$ is an isometric embedding. This is a consequence of the {\it McShane extension theorem} \cite[Theorem~1.33]{Weaver}, which allows one to extend any 1-Lipschitz function $X\to\R$ to a 1-Lipschitz function $Y\to\R$. See \cite[Chapter~3]{Weaver} for more background on free spaces (and note that they are referred to as \emph{Arens--Eels spaces} in that text, and in other places {\it Kantorovich--Rubinstein normed spaces} and {\it transportation cost spaces}).

We also need to consider Lipschitz free and Wasserstein-1 spaces over semimetric spaces. Let $X$ be a set equipped with a semimetric $d$. We define the 1-Lipschitz functions as usual: $f: X \to \R$ is 1-Lipschitz if $|f(x)-f(y)| \leq d(x,y)$ for all $x,y\in X$. The most important examples of 1-Lipschitz functions we need are distance-to-sets functions: $X \ni x \mapsto \dist(A,x) \in \R$, where $A \sbs X$ and $\dist(A,x) := \inf\{d(a,x): a\in A\}$. This function is 1-Lipschitz simply by the triangle inequality of $d$, positive definiteness is not required.

We require the following reduction to linear maps, which was essentially proved by Naor--Schechtman in \cite[Lemma~3.1]{NS}, in order to study embedding of Wasserstein metrics into $L_1$. When $Z$ is a semimetric space, we call the inifimal $L$ for which there exists an $L$-biLipschitz embedding $Z\to L_1$ the {\it $L_1$-distortion} of $Z$. When $E$ is a seminormed space, we call the inifimal $L$ for which there exists $m\in\N$ and an $L$-isomorphic embedding $E\to\ell_1^m$ the {\it linear $\{\ell_1^m\}_{m=1}^\infty$-distortion} of $E$.

\begin{lemma}[Reduction to Linear Maps] \label{lem:linearreduction}
For every finite semimetric space $X$, the $L_1$-distortion of $W_1(X)$ is equal to the linear $\{\ell_1^m\}_{m=1}^\infty$-distortion of $\LF(X)$.
\end{lemma}

\begin{proof}
Let $(X,d)$ be a finite semimetric space. The case where $X$ is a metric space was essentially proved in \cite[Lemma~3.1]{NS}, and one can find a more general statement and proof in \cite[\S5]{GO}. We will explain how to reduce to the metric case.

Let $\sim$ denote the equivalence relation on $X$ given by $x\sim y$ if $d(x,y)=0$. Then $d$ descends to a metric on $X/\sim$. The canonical projection $X \to X/\sim$ induces surjective isometries $\LF(X) \to \LF(X/\sim)$ and $W_1(X) \to W_1(X/\sim)$, with the former map being linear. Hence, the linear $\{\ell_1^m\}_{m=1}^\infty$-distortion of $\LF(X)$ is the same as that of $\LF(X/\sim)$, and the $L_1$-distortion of $W_1(X)$ is the same as that of $W_1(X/\sim)$. Since these two quantities agree on $X/\sim$ by \cite[\S5]{GO}, they agree on $X$.
\end{proof}

We need one more basic fact. Suppose that $X$ is a finite semimetric space, $\lambda: \LF(X) \to \R$ is a linear functional, and $x_0\in X$ is some chosen basepoint. Define $f: X\to\R$ by $f(x) := \lambda(\delta_x-\delta_{x_0})$. Then for any $\mu \in \LF(X)$, we have that $\lambda(\mu)=\int_X fd\mu$. This follows from the fact the two sides of the equation agree when $\mu = \delta_x-\delta_{x_0}$ for some $x\in X$, and signed measures of this form clearly span $\LF(X)$. We say that $f$ {\it represents} the functional $\lambda$.

\subsection{Stochastic embeddings} \label{ss:stochasticembeddings}
Suppose $(X,x_0)$ and $(Y,y_0)$ are pointed sets. Following \cite[\S3.1]{Gartland}, a \emph{random map} (of finite type) from $X$ to $Y$ is a collection of maps $\phi = \{\phi_i: M \to N\}_{i=1}^k$ equipped with some probability measure $\P$ such that $\phi_i(x_0) = y_0$ for all $i\in\{1,\dots k\}$. If $X,Y$ are equipped with metrics $d_X,d_Y$, then a random map is {\it noncontractive} if, for every $i\in\{1,\dots k\}$ and $x,y \in X$, $d_Y(\phi_i(x),\phi_i(y)) \geq d_X(x,y)$.

We define the \emph{stochastic distortion function} of a random noncontractive map $(\phi,\P)$ to be the function $D_\phi: X\times X \times [1,\infty) \to [0,1]$ defined by
\begin{equation*}
    D_\phi(x,y)(t) := \P(\{\phi_i:d_Y(\phi_i(x),\phi_i(y)) \geq td_X(x,y)\}).
\end{equation*}
For $p \in (0,\infty)$, the \emph{$L_p$-stochastic distortion} of the random noncontractive map $(\phi,\P)$ is the quantity
$$\|\phi\|_p^p := \sup_{x,y\in X}\int_{0}^\infty pt^{p-1}D_\phi(x,y)(t)dt,$$
and the \emph{weak-$L_p$-stochastic distortion} is the quantity
$$\|\phi\|_{w,p}^p := \sup_{x,y\in X}\sup_{t\geq 1}t^pD_\phi(x,y)(t).$$
We say that a random noncontractive map between metric spaces is an {\it $L_p$-stochastic embedding} if it has finite $L_p$-stochastic distortion and a {\it weak-$L_p$-stochastic embedding} if it has finite weak-$L_p$-stochastic distortion.

Clearly, we have that if $(\phi,\P)$ is a random noncontractive map, then $\|\phi\|_1$ is equal to the least constant $L$ such that
\begin{equation} \label{eq:stochasticembedding}
    \E[d_Y(\phi(x),\phi(y))] = \sum_i d_Y(\phi_i(x),\phi_i(y))\P(\phi_i) \leq Ld_X(x,y)
\end{equation}
for all $x,y\in X$. Random noncontractive maps satisfying the above inequality have been called  {\it stochastic embeddings} or {\it stochastic biLipschitz embeddings}, and $L$ is called the {\it stochastic distortion} of $\phi$. Thus, \eqref{eq:stochasticembedding} says that the stochastic distortion of $\phi$ is the same as its $L_1$-stochastic distortion $\|\phi\|_1$. Stochastic embeddings were introduced in \cite{Bartal} (not under this name), and their relationship to embeddings of Wasserstein-1 spaces into $L_1$ was realized in \cite{Charikar} (see also \cite{IT03}).

Before proceeding, it will be useful to reorient our view of random maps into metric spaces. Suppose $\mathcal{C}$ is a class of finite, pointed metric spaces with the following closure property: for any finite collection $\{Y_i\}_{i=1}^k \sbs \mathcal{C}$, there exist $Y \in \mathcal{C}$ and, for each $1\leq i\leq k$, a basepoint-preserving isometric embedding $Y_i \hookrightarrow Y$. If $X$ is a finite metric space, we define a {\it random map into the class $\mathcal{C}$} to be a collection of maps $\phi = \{\phi_i: X \to Y_i\}_{i=1}^k$ equipped with some probability measure $\P$, where each $Y_i$ is an element of $\mathcal{C}$. Note that the target spaces are allowed to be variable and that there is no basepoint-preservation requirement. Noncontractivity of $\phi$ and its stochastic distortion function $D_\phi$ are defined the same way as before. Due to the assumption on the closure properties of $\mathcal{C}$, we can see that for any random noncontractive map $\phi$ from $X$ into $\mathcal{C}$, if we equip $X$ with any basepoint, then there exists $Y\in\mathcal{C}$ and a random (basepoint-preserving) noncontractive map $X \to Y$ with the same distortion function as $\phi$. Hence, in the case when this closure property is satisfied, there is no loss in generality of allowing variable target spaces and disregarding basepoint-preservation. The class of finite, pointed ultrametric spaces obviously has the closure property, because the $\ell^\infty$-product of a finite collection of finite, pointed ultrametric spaces is again a finite, pointed ultrametric space, and it admits the obvious basepoint-preserving isometric embeddings from its factors. The only random maps we will consider in this article are those into the class of ultrametric spaces.

\begin{definition}[$L_1$-Stochastic Embedding into Ultrametric Spaces] \label{def:stochasticembedding}
We say that a metric space $X$ admits an {\it $L_p$-stochastic embedding (resp. weak-$L_p$-stochastic embedding) into the class of ultrametric spaces} if there exists $D<\infty$ such that every finite subset of $X$ admits an $L_p$-stochastic embedding (resp. weak-$L_p$-stochastic embedding) into the class of ultrametric spaces with $L_p$-stochastic distortion (resp. weak-$L_p$-stochastic distortion) $\leq D$.
\end{definition}

The following two results can be quickly deduced by combining results from \cite{Kalton} and \cite{Gartland}).

\begin{lemma} \label{lem:LF1}
Let $X$ be a finite metric space that stochastically embeds into the class of ultrametric spaces with $L_1$-stochastic distortion $\leq D$. Then $\LF(X)$ is $D'$-isomorphic to a $C'$-complemented subspace of $\ell_1$, where $D',C'<\infty$ depend only on $D$.
\end{lemma}

\begin{proof}
Since ultrametric spaces have Nagata dimension 0 with constant $\gamma = 1$ (this follows quickly from the definitions -- one can take all balls of radius $r$ as the cover $\B$), we have by \cite[Corollary~3.2]{Gartland} that $\LF(X)$ is $D$-isomorphic to a $C$-complemented subspace of $L_1(\Omega;\LF(U))$ for some finite ultrametric space $U$, where $\Omega$ is the finite probability space $(\{1,\dots k\},\P)$ underlying the stochastic embedding, and $C$ depends only on $D$. Since $\LF(U)$ is 8-isomorphic to the finite dimensional space $\ell_1^{|U|-1}$ (see, for example, \cite[Corollary~2.5]{DKO}), the conclusion follows.
\end{proof}

\begin{theorem} \label{thm:stochasticembedding=>LF}
Suppose $X$ is an infinite metric space such that all of its bounded subsets are finite. If $X$ admits an $L_1$-stochastic embedding into the class of ultrametric spaces, then $\LF(X)$ is isomorphic to $\ell_1$.
\end{theorem}

\begin{proof}
Assume that $X$ admits an $L_1$-stochastic embedding into the class of ultrametric spaces. Let $D<\infty$ such that every finite subset of $X$ admits an $L_1$-stochastic embedding into the class of ultrametric spaces with $L_1$-stochastic distortion $\leq D$. By \cite[Lemma~4.2]{Kalton}, there is a sequence of bounded (and hence finite) subsets $(X_n)_{n=1}^\infty$ of $X$ such that $\LF(X)$ is isomorphic to a complemented subspace of the $\ell_1$-sum $\oplus^1_n \LF(X_n)$. By Lemma~\ref{lem:LF1}, $\LF(X_n)$ is $D'$-isomorphic to a $C'$-complemented subspace of $\ell_1$ for each $n\in\N$, where $D',C'<\infty$ depend only on $D$. Thus, the $\ell_1$-sum $\oplus^1_n \LF(X_n)$ is $D'$-isomorphic to a $C'$-complemented subspace of $\ell_1$, and so $\LF(X)$ is isomorphic to a complemented subspace of $\ell_1$. Since every infinite dimensional, complemented subspace of $\ell_1$ is isomorphic to $\ell_1$, (see, for example, \cite[Theorem~15]{AO}), $\LF(X)$ is isomorphic to $\ell_1$.
\end{proof}

Recall the definition of uniformly Lipschitz actions with compression exponent 1 from \S\ref{sec:intro}. The following corollary proves Theorem~\ref{thm:nearlyundistortedorbit}.

\begin{corollary} \label{cor:nearlyundistortedorbit}
Let $\Gamma$ be a finitely generated infinite group with finite Nagata dimension. Then $\Gamma$ admits uniformly Lipschitz actions on $\ell_1$ with compression exponent 1.
\end{corollary}

\begin{proof}
Let $\eps\in(0,1)$. Since the assignment $X \mapsto \LF(X)$ is functorial, we have that $\Gamma$ acts by affine isometries on $\LF(\Gamma,d_\Gamma^{1-\eps})$ with $\|0 - \gamma\cdot 0\|_{\LF(\Gamma,d_\Gamma^{1-\eps})} = d_\Gamma(e,\gamma)^{1-\eps}$ for all $\gamma\in\Gamma$ (see \cite[Lemma~6.1]{Ver2} for details). By \cite[Theorem~B]{Gartland}, $(\Gamma,d_\Gamma^{1-\eps})$ admits an $L_1$-stochastic embedding into the class of ultrametric spaces, and then by Theorem~\ref{thm:stochasticembedding=>LF}, $\LF(\Gamma,d_\Gamma^{1-\eps})$ is isomorphic to $\ell_1$. Conjugating the action of $\Gamma$ on $\LF(\Gamma,d_\Gamma^{1-\eps})$ with an isomorphism $\LF(\Gamma,d_\Gamma^{1-\eps}) \approx \ell_1$ yields the desired action on $\ell_1$.
\end{proof}

\subsection{Conditionally negative definite kernels and actions on $\boldsymbol{E\subseteq L_1}$}
Let $X$ be a set, and let $\Phi:X\times X\to[0,\infty)$ be a symmetric function such that $\Phi(x,x)=0$ for all $x\in X$. We say that $\Phi$ is a {\it conditionally negative definite (CND) kernel} on $X$ if
\begin{align*}
	\sum_{i,j=1}^n \alpha_i\alpha_j\Phi(x_i,x_j) \leq 0
\end{align*}
for every choice of elements $x_1,\ldots,x_n\in X$ and scalars $\alpha_1,\ldots,\alpha_n\in\R$ satisfying
\begin{align*}
	\sum_{i=1}^n\alpha_i=0.
\end{align*}

The following example will be particularly relevant for our purposes. Let $f:X\to L_1$ be any map. Then the kernel $\Phi:X\times X\to[0,\infty)$ defined by
\begin{align*}
	\Phi(x,y)=\|f(x)-f(y)\|_1
\end{align*}
is CND; see e.g. \cite[Lemma~7]{Now} for a proof.

Let $\Gamma$ be a countable discrete group. Then $\Gamma$ can be endowed with a proper, left-invariant metric $d_\Gamma$, and this metric is unique up to coarse equivalence; see e.g. \cite[Proposition 2.3.3]{Willett}. Let $\Phi:\Gamma\times \Gamma\to[0,\infty)$ be a CND kernel. We say that $\Phi$ is (multiplicatively) almost invariant if there is a constant $C<\infty$ such that, for all $g,x,y\in\Gamma$,
\begin{align*}
	\Phi(gx,gy)\leq C\Phi(x,y).
\end{align*}
We say that such a kernel is proper if $\Phi(e,\gamma)\to\infty$ as $d_\Gamma(e,\gamma)\to\infty$. Here $e$ denotes the identity element of $\Gamma$. The following result was essentially proved in \cite[Theorem 1.1]{Ver}.

\begin{theorem}\label{thm:AICNDK}
Let $\Gamma$ be a countable, discrete group equipped with a proper, left-invariant metric. Then $\Gamma$ admits a proper, uniformly Lipschitz affine action on a subspace of $L_1$ if and only if $\Gamma$ admits a proper, almost invariant CND kernel.
\end{theorem}

\begin{proof}
This equivalence was proved in \cite[Theorem 1.1]{Ver} for a stronger notion of almost invariance. We only need to show that a multiplicatively almost invariant CND kernel gives rise to an additively almost invariant CND kernel, and then apply \cite[Theorem 1.1]{Ver}. Let $\Phi_0$ be a proper multiplicatively almost invariant CND kernel on $\Gamma$, and let us define, for all $x,y\in\Gamma$,
	\begin{align*}
		\Phi(x,y)=\log\left(1+\Phi_0(x,y)\right).
	\end{align*}
It holds that $\Phi$ is a CND kernel too (see \cite[Lemma~2.2]{Ver} for an argument). Moreover, it is proper because $\Phi_0$ is proper. Finally, since $\Phi_0$ is (multiplicatively) almost invariant, there is $C<\infty$ such that, for all $g,x,y\in\Gamma$,
	\begin{align*}
		\Phi(gx,gy) &\leq \log\left(1+C \Phi_0(x,y)\right)\\
		&\leq \log(C) + \log\left(1+\Phi_0(x,y)\right)\\
		&= \log(C) + \Phi(x,y),
	\end{align*}
	which shows that $\Phi$ is additively almost invariant.
\end{proof}

\section{The Rescaling Lemma} \label{sec:rescaling}
In this section, we prove and apply a trick that we call the {\it Rescaling Lemma}. The trick is completely elementary and has appeared in coarse geometry before, for example to provide a simple proof that spaces of finite asymptotic dimension coarsely embed into a finite product of trees \cite{CoarseEmbeddingIntoTrees}. We are able to get a surprising amount of mileage out of the Rescaling Lemma, providing in Proposition~\ref{prop:dictionary} a dictionary between coarse and Lipschitz geometric properties, which in turn powers some of the main results of the article.

\begin{lemma} [The Rescaling Lemma]\label{lem:rescaling}
Let $\rho: [1,\infty) \to [0,\infty)$ be any function mapping bounded sets to bounded sets (for example, if $\rho$ is nondecreasing). Then there exists a proper gauge $\omega: [0,\infty) \to [0,\infty)$ such that $\omega(\rho(t)) \leq 4\omega(t)$ for all $t \in [1,\infty)$.   
\end{lemma}

\begin{proof}
Define $\tilde{\rho}: [1,\infty) \to [0,\infty)$ by $\tilde{\rho}(t) := \sup_{s\in[1,t]}\{\rho(s),3t\}$. The supremum exists by assumption on $\rho$, and clearly $\rho(t) \leq \tilde{\rho}(t)$ with $\tilde{\rho}$ nondecreasing. Define a sequence of points $\{x_k\}_{k=1}^\infty \sbs [0,\infty)$ recursively by
\begin{align*}
    x_1 := 1, && x_{k+1} := \tilde{\rho}(x_k).
\end{align*}
The definition of $\tilde{\rho}$ implies $x_{k+1} \geq 3x_k$.
In particular, $x_k$ strictly increases to $\infty$. Define a function $\omega: \{0\} \cup \{x_k\}_{k=1}^\infty \to [0,\infty)$ by $\omega(0) := 0$ and $\omega(x_k) := 2^k$. Since $\{x_k\}_{k=1}^\infty$ is increasing, $\omega$ is well-defined and strictly increases to $\infty$ as well. Extend $\omega$ to a function on all of $[0,\infty)$ by linear interpolation on the intervals $[0,x_1]$ and $[x_k,x_{k+1}]$. This extension is continuous and also strictly increases to $\infty$, and thus is a homeomorphism. It will be concave as long as the slope of $\omega$ on adjacent intervals is decreasing. Comparing the slopes on the first two intervals $[0,x_1], [x_1,x_2]$, we see that
$$\frac{\omega(x_2)-\omega(x_1)}{x_2-x_1} = \frac{4-2}{x_2-1} \leq \frac{4-2}{3-1} = 1 < 2 = \frac{\omega(x_1)-\omega(0)}{x_1-0}.$$
Comparing the slopes on the other adjacent intervals, we see that
$$\frac{\omega(x_{k+1})-\omega(x_k)}{x_{k+1}-x_k} = \frac{2^k}{x_{k+1}-x_k} \leq \frac{2^k}{3x_k-x_k}$$
$$=\frac{2^{k-1}}{x_k} < \frac{2^{k-1}}{x_k-x_{k-1}} = \frac{\omega(x_k)-\omega(x_{k-1})}{x_k-x_{k-1}}.$$
It remains to check that $\omega(\rho(t)) \leq 4\omega(t)$.

Let $t \in [1,\infty)$. Choose $m \in \N$ such that $x_{m} \leq t < x_{m+1}$. Using the facts that $\omega$ and $\tilde{\rho}(t)$ are nondecreasing and that $\rho(t) \leq \tilde{\rho}(t)$, we have
\begin{align*}
    \omega(\rho(t)) \leq \omega(\tilde{\rho}(t)) \leq \omega(\tilde{\rho}(x_{m+1})) = \omega(x_{m+2}) = 2^{m+2} = 4\omega(x_m) \leq 4\omega(t).
\end{align*}
\end{proof}

The Rescaling Lemma allows us to provide a dictionary between common notions in coarse and Lipschitz geometry. We do so for three such notions in Proposition~\ref{prop:dictionary}. First, we need a standard lemma about modifying coarse embeddings into Banach spaces.

\begin{lemma} \label{lem:lowergauge}
Let $(X,d)$ be a countable, 1-separated metric space that coarsely embeds into an infinite dimensional Banach space $E$. Then there exists a coarse embedding from $X$ into $E$ that admits a lower control function $\rho_1$ that is a proper gauge satisfying $\rho_1(1) = 1$.
\end{lemma}

\begin{proof}
Let $f: X \to E$ be a coarse embedding with nondecreasing lower and upper control functions $\rho_1$ and $\rho_2$. Choose $t_0 \in [0,\infty)$ such that $\rho_1(t_0) \geq 4$. Choose a maximal $t_0$-separated subset $N \sbs X$, so that $f(N)$ is 4-separated in $E$. Since $E$ is infinite-dimensional and $X$ is countable and 1-separated, it is possible to modify $f$ on $X\setminus N$ in order to find a new coarse embedding $\tilde{f}: X \to E$ with lower control function $\tilde{\rho}_1$ satisfying $\tilde{\rho}_1(1) \geq 1$. We will quickly sketch the construction: Let $\pi: X \to N$ be any nearest neighbor projection, and form the partition $X = \sqcup_{x\in N} \pi^{-1}(x)$ into the fibers of $\pi$. Each fiber $\pi^{-1}(x)$ is contained in the ball $B_{t_0}(x)$, and hence its image under $f$ is contained in $B_{\rho_2(t_0)}(x)$. For each $x \in N$, Riesz's lemma allows us to choose a mapping $\tilde{f}_x$ from $\pi^{-1}(x)$ to $B_{4/3}(f(x)) \sbs E$ such that $\tilde{f}_x(x) = f(x)$ and $\tilde{f}_x(\pi^{-1}(x))$ is 1-separated. Pasting these maps together, we get a map $\tilde{f}: X \to E$ that agrees with $f$ on $N$, has image that is 1-separated, and satisfies $\|\tilde{f}(x)-\tilde{f}(y)\| \leq \|f(x)-f(y)\| + 2(\rho_2(t_0)+\frac{4}{3})$. This completes the construction. Since $X$ and $\tilde{f}(X)$ are 1-separated, $\tilde{f}$ admits a lower control function $\tilde{\rho}_1$ with $\tilde{\rho}_1(1) \geq 1$. Henceforth, we will keep writing $f$ and $\rho_1$ instead of $\tilde{f}$ and $\tilde{\rho}_1$.

One may easily see how to construct a proper gauge serving as the lower control function. We provide short proof using the Rescaling Lemma.  Let $\rho_1^{-1}: [1,\infty) \to [0,\infty)$ be the pseudoinverse of $\rho_1$ defined by $\rho_1^{-1}(s) := \inf\{t\in[1,\infty): \rho_1(t) > s\}$. This function is well-defined since $\rho_1(t)\to\infty$, is obviously nondecreasing, and satisfies $t \leq \rho_1^{-1}(\rho_1(t))$ for all $t \geq 0$ since $\rho_1$ is nondecreasing. We apply the Rescaling Lemma~\ref{lem:rescaling} to $\rho_1^{-1}$ and obtain a proper gauge $\omega: [0,\infty)\to[0,\infty)$ satisfying $\omega(\rho_1^{-1}(t)) \leq 4\omega(t)$ for every $t \geq 1$. Then, for all $x\neq y\in X$, it holds that $\rho_1(d(x,y)) \geq \rho_1(1) \geq 1$, and thus we have
\begin{align*}
    \omega(d(x,y)) &\leq \omega(\rho_1^{-1}(\rho_1(d(x,y)))) \\
    &\leq 4\omega(\rho_1(d(x,y))) \\
    &\leq 4\omega(1)\rho_1(d(x,y)) \\
    &\leq 4\omega(1)\|f(x)-f(y)\|.
\end{align*}
Hence, the proper gauge $\omega_0 := \omega/\omega(1)$ is a lower control function with $\omega_0(1)=1$ for the coarse embedding $4f$.
\end{proof}

We are now ready to prove our dictionary.

\begin{proposition}[Coarse-to-Lipschitz Dictionary] \label{prop:dictionary}
Let $n$ be a nonnegative integer, $p\in[1,2]$, and consider the following coarse ($C_1$),($C_2$),($C_3$) and Lipschitz ($L_1$),($L_2$),($L_3$) properties of a metric space $Z$:
\begin{multicols}{2}
\begin{itemize}
    \item[($C_1$)] $Z$ has asymptotic dimension $n$.
    \item[($C_2$)] $Z$ has controlled separation.
    \item[($C_3$)] $Z$ coarsely embeds into $L_p$.
    \item[($L_1$)] $Z$ has Nagata dimension $n$.
    \item[($L_2$)] $Z$ has finite separation modulus.
    \item[($L_3$)] $Z$ biLipschitzly embeds into $L_p$.
\end{itemize}
\end{multicols}
\noindent Let $i\in\{1,2,3\}$ and let $(X,d)$ be a countable, uniformly discrete metric space. Then the following are equivalent:
\begin{enumerate}
    \item $(X,d)$ satisfies ($C_i$).
    \item There exists a proper gauge $\omega$ such that $(X,\omega\circ d)$ satisfies ($L_i$).
    \item There exists a metric space $Y$ coarsely containing $(X,d)$ that satisfies ($L_i$).
\end{enumerate}
\end{proposition}

\begin{proof}
The implication $(2) \implies (3)$ follows from the fact that the identity map $(X,d) \to (X,\omega \circ d)$ is a coarse equivalence whenever $\omega$ is a proper gauge. The implication $(3) \implies (1)$ follows from the facts that ($L_i$) implies ($C_i$) and that ($C_i$) is inherited under coarse embeddings. We will prove $(1)\implies (2)$. We will only provide details for the cases $i=2,3$, as the case $i=1$ follows from an argument similar to that for $i=2$. By multiplying $d$ by a constant, we may assume that $(X,d)$ is 1-separated.

\fbox{$i=2$}: Assume that $X$ has controlled separation. Let $\Delta: (0,1]\to[0,\infty)$ be a nonincreasing separation control function (as mentioned in \S\ref{sec:prelims}, we may take the domain of $\Delta$ to be $(0,1]$ here since $X$ is 1-separated). Apply the Rescaling Lemma~\ref{lem:rescaling} to the nondecreasing function $\rho(t) = \Delta(t^{-1})$ to obtain a proper gauge $\omega: [0,\infty) \to [0,\infty)$ with $\omega(\Delta(t^{-1})) \leq 4\omega(t)$ for all $t \in [1,\infty)$. We will verify that the separation modulus $\sigma$ of $(X,\omega \circ d)$ is at most 4.

Let $\eps \in (0,1]$. By dividing $\omega$ by $\omega(1)$, we may assume that $\omega(1)=1$. Set $t := \frac{1}{\omega^{-1}(\eps^{-1})}$, and note that $t \leq 1$ since $\eps\leq 1$, $\omega(1)=1$, and $\omega$ is increasing. Choose a Radon probability measure $\P$ on the set of $\Delta(t)$-bounded partitions of $(X,d)$ such that
\begin{equation} \label{eq:randompartition}
    \P_P(P(x)\neq P(y)) \leq t d(x,y)
\end{equation}
for every $x,y \in X$. Then $\P$ is also a Radon probability measure on the set of $D$-bounded partitions of the 1-separated metric space $(X,\omega \circ d)$, where
$$D = \omega(\Delta(t)) \leq 4\omega(t^{-1}) = \frac{4}{\eps}.$$

Fix $x\neq y \in X$. Then, using the facts that $s\omega(t) \leq \omega(st)$ for all $s \leq 1$ and that $\omega$ is increasing, we have
\begin{align*}
    \P_P(P(x)\neq P(y)) &= \eps\P_P(P(x)\neq P(y))\eps^{-1} \\
    &= \eps\P_P(P(x)\neq P(y))\omega(t^{-1}) \\
    &\leq \eps\omega(\P_P(P(x)\neq P(y))t^{-1}) \\
    &\overset{\eqref{eq:randompartition}}{\leq} \eps\omega(d(x,y)).
\end{align*}
This completes the proof for $i=2$.

\fbox{$i=3$}: We first prove the case $p=2$ and then explain how to extend to general $p\in[1,2]$. Assume that $(X,d)$ coarsely embeds into $L_2$. By Lemma~\ref{lem:lowergauge}, there is a function $f: X \to L_2$, a proper gauge $\rho_1: [0,\infty)\to[0,\infty)$ with $\rho_1(1)= 1$, and a nondecreasing function $\rho_2: [0,\infty) \to [0,\infty)$ such that
$$\rho_1(d(x,y)) \leq \|f(x)-f(y)\|_2 \leq \rho_2(d(x,y))$$
for all $x,y \in X$. Apply the Rescaling Lemma~\ref{lem:rescaling} to $\rho_2\circ\rho_1^{-1}$ to obtain a proper gauge $\omega$ with $\omega(\rho_2(\rho_1^{-1}(t))) \leq 4\omega(t)$ for every $t \geq 1$. Then, for every $x\neq y\in X$, we have
\begin{equation*}
    \omega(\rho_1(d(x,y))) \leq \omega(\|f(x)-f(y)\|_2) \leq \omega(\rho_2(d(x,y)))
\end{equation*}
and
\begin{equation*}
    \omega(\rho_2(d(x,y))) = \omega(\rho_2(\rho_1^{-1}(\rho_1(d(x,y))))) \leq 4\omega(\rho_1(d(x,y))).
\end{equation*}
The above inequalities imply that $(X,(\omega\circ\rho_1)\circ d)$ admits a 4-biLipschitz embedding into $(L_2,\omega(\|\cdot\|_2))$, and this latter space itself biLipschitzly embeds into $L_2$ by \cite[Remark~5.4]{MNquotients}. This completes the proof for the case of $p=2$. The proof for general $p\in[1,2]$ follows from the case $p=2$ and the following two facts (see \S\ref{ss:Banach}): (i) $L_2$ isometrically embeds into $L_p$ and (ii) $(L_p,\|\cdot\|_2^{p/2})$ isometrically embeds into $L_2$.
\end{proof}

\section{Coarse Embeddability Implies Equivariant Coarse Embeddability} \label{sec:equivariantembeddability}

In this section, we prove Theorem~\ref{thm:embedding=>equivariantembedding}: a group having a proper, uniformly Lipschitz, affine action on a subspace of $L_1$ is equivalent to the seemingly weaker property of admitting a coarse embedding into $L_1$.

Recall that, if $E$ is a subspace of $L_1$, then the kernel $X \times X \ni (x,y)\mapsto \|f(x)-f(y)\|_1 \in E \sbs L_1$ is CND. If, in addition, this kernel is almost invariant, then one can use Theorem~\ref{thm:AICNDK} in order to construct an affine action. However, one should not expect this strong condition to be satisfied for a general coarse embedding. The way to remedy this situation is to apply the Coarse-to-Lipschitz Dictionary (Proposition~\ref{prop:dictionary}) in order to introduce a sufficient amount of invariance in order to apply Theorem~\ref{thm:AICNDK}. Before proving Theorem~\ref{thm:embedding=>equivariantembedding} in full generality, we give a simplified argument in the case where $\Gamma$ embeds into $L_1$ with positive compression exponent. In this case, the proof is more self-contained (relies only on Theorem~\ref{thm:AICNDK} and not on Proposition~\ref{prop:dictionary}), and one can obtain quantitative lower bounds on the properness of the corresponding CND kernel (Remark~\ref{rmk:orbitdistortion}).

\subsection{A simple argument for groups with positive compression exponent}

We begin by illustrating our argument in a particular case allowing for a simplified proof. Following \cite{NaoPer}, we define the $L_1$-compression exponent of a separable metric space $(X,d)$ as the supremum of all $\alpha\in[0,1]$ such that there is a Lipschitz map $f:X\to L_1$, and a constant $c>0$ such that
\begin{align*}
	\|f(x)-f(y)\|_1\geq c d(x,y)^\alpha
\end{align*}
for all $x,y\in X$. We denote this exponent by $\alpha_1^*(X)$. If $\Gamma$ is a finitely generated group, we define $\alpha_1^*(\Gamma)$ to be $\alpha_1^*(\Gamma,d_\Gamma)$ where $d_\Gamma$ is a word metric.

\begin{proposition}\label{prop:positive_ce}
Let $\Gamma$ be a finitely generated group with $L_1$-compression exponent $\alpha_1^*(\Gamma)>0$. Then $\Gamma$ admits a proper, almost invariant CND kernel.
\end{proposition}
\begin{proof}
Let $\alpha\in(0,\alpha_1^*(\Gamma))$, and let $f:\Gamma\to L_1$ be a Lipschitz map such that, for all $x,y\in\Gamma$,
\begin{align*}
	\|f(x)-f(y)\|_1 \geq c d_\Gamma(x,y)^\alpha,
\end{align*}
where $c>0$ is a constant. Since $f$ is Lipschitz, there is another constant $C>0$ such that
\begin{align*}
	\|f(x)-f(y)\|_1 \leq C d_\Gamma(x,y).
\end{align*}
Let us define a kernel $\Phi:\Gamma\times\Gamma\to[0,\infty)$ by
\begin{align*}
	\Phi(x,y)=\log(	\|f(x)-f(y)\|_1+1)
\end{align*}
for all $x,y\in\Gamma$. Since $f$ is a coarse embedding into $L_1$, $(x,y) \mapsto \|f(x)-f(y)\|_1$ is a proper CND kernel, and thus $\Phi$ is also a proper CND kernel (see \cite[Lemma~2.2]{Ver} for details). Moreover, for all $g,x,y\in\Gamma$,
\begin{align} \label{eq:compression}
	\Phi(gx,gy) = \log(\|f(gx)-f(gy)\|_1+1) \leq  \log(C d_\Gamma(x,y)+1).
\end{align}
Now fix $\beta>1$ large enough so that
\begin{align*}
	Ct+1\leq\left(ct^\alpha+1\right)^\beta
\end{align*}
for all $t>0$. Then \eqref{eq:compression} yields
\begin{align*}
	\Phi(gx,gy) &\leq  \log\left(\left(cd_\Gamma(x,y)^\alpha+1\right)^\beta\right)\\
	&\leq \beta\log(\|f(x)-f(y)\|_1+1)\\
	&= \beta\Phi(x,y).
\end{align*}
This shows that $\Phi$ is almost invariant.
\end{proof}

\begin{remark} \label{rmk:orbitdistortion}
Inspecting the proof, one can see that the almost invariant CND kernel $\Phi$ constructed in the proof above satisfies $\Phi(e,\gamma) \gtrsim \log (d_\Gamma(e,\gamma))$.
\end{remark}

\subsection{Proof of Theorem~\ref{thm:embedding=>equivariantembedding}} \label{ss:embedding=>equivariantembedding}
We now prove Theorem~\ref{thm:embedding=>equivariantembedding}, which states that if $\Gamma$ is a countable discrete group, then it admits a proper, uniformly Lipschitz action on a subspace of $L_1$ if and only if it coarsely embeds into $L_1$.

\begin{proof}[Proof of Theorem~\ref{thm:embedding=>equivariantembedding}]
The ``only if" direction is trivial, because if $\Gamma$ admits such an action on $E \sbs L_1$, then the orbit map $\Gamma \ni \gamma \to \gamma\cdot 0 \in E \sbs L_1$ is a coarse embedding.

For the ``if" direction, assume that $(\Gamma,d_\Gamma)$ coarsely embeds into $L_1$. The discreteness of $\Gamma$ and left-invariance of $d_\Gamma$ imply that $(\Gamma,d_\Gamma)$ is uniformly discrete. Thus, by Proposition~\ref{prop:dictionary}, there is a proper gauge $\omega$ and a biLipschitz embedding $f: (\Gamma,\omega \circ d_\Gamma) \to L_1$. By dilating the map, we may assume that there is a $C<\infty$ such that
$$\omega(d(x,y)) \leq \|f(x)-f(y)\|_1 \leq C\omega(d(x,y))$$
for all $x,y\in\Gamma$. Then the proper CND kernel $\Phi(x,y) := \|f(x)-f(y)\|_1$ is almost invariant. Indeed, for any $g,x,y\in \Gamma$, we have
$$\Phi(gx,gy) \leq C\omega(d(gx,gy)) = C\omega(d(x,y)) \leq C\Phi(x,y).$$
Hence, by Theorem~\ref{thm:AICNDK}, $\Gamma$ admits a proper, uniformly Lipschitz affine action on a subspace of $L_1$.
\end{proof}

\section{Property A, Separation Modulus, and Stochastic Embeddings} \label{sec:PropA}

In this section, we prove the implications $(1)\implies(2)\implies(3)$ in Theorem \ref{thm:mainextended}.

\subsection{Property A and separation modulus}
The implication $(1)\implies(2)$ in Theorem~\ref{thm:mainextended} is quite simple with the help of Elek's characterization of Property A (Theorem~\ref{thm:Elek}) and the Coarse-to-Lipschitz Dictionary (Proposition~\ref{prop:dictionary}).

\begin{theorem} \label{thm:PropA=>separation}
Let $X$ be a connected, bounded degree graph and $d$ the shortest path metric on $X$. If $X$ has Property A, then there exists a proper gauge $\omega$ such that $\LF(X,\omega\circ d)$ has finite separation modulus.
\end{theorem}

\begin{proof}
Assume that $X$ has Property A. Then $X$ is locally strongly hyperfinite by Theorem~\ref{thm:Elek}. Then $(X,d)$ has controlled separation by Lemma~\ref{lem:controlledseparation}. Then by Proposition~\ref{prop:dictionary}, there exists a proper gauge $\omega$ such that $(X,\omega \circ d)$ has finite separation modulus.
\end{proof}

\subsection{Separation modulus and stochastic embeddings}

A strong connection between separation modulii and stochastic embeddings into ultrametric spaces was already clear and used in their first appearance \cite{Bartal}. Using the language of $L_p$- and weak-$L_p$-stochastic distortion, we provide two lemmas that, we believe, makes their connection even more clear. Clearly, by Markov's inequality, for any random noncontractive map $\phi$, we have that $\|\phi\|_{w,p} \leq \|\phi\|_p$. Of course, the reverse inequality $\|\phi\|_{p} \leq C\|\phi\|_{w,p}$ is not true with any uniform constant $C<\infty$. However, it does become true after snowflaking the metrics. This is an elementary and purely measure-theoretic fact which we prove in the next lemma.

\begin{lemma} \label{lem:weaksnowflake}
Suppose $p\in(0,\infty)$ and $(\phi,\P)$ is a random noncontractive map between finite, pointed metric spaces $(X,d_X),(Y,d_Y)$ with weak-$L_p$-stochastic distortion $D <\infty$. Then for each $\eps \in (0,1)$, the same random map $(\phi,\P)$ is a random noncontractive map from $(X,d_X^{1-\eps})$ into $(Y,d_Y^{1-\eps})$ with $L_p$-stochastic distortion at most $D/\eps^{1/p}<\infty$. In particular, if a finite metric space $(X,d_X)$ admits a weak-$L_p$-stochastic embedding into the class of ultrametric spaces with weak-$L_p$-stochastic distortion $\leq D$, then $(X,d_X^{1-\eps})$ admits a stochastic embedding into the class of ultrametric spaces with $L_p$-stochastic distortion $< D/\eps^{1/p}$.
\end{lemma}

\begin{proof}
Let $\eps\in(0,1)$. First, quickly notice that the second conclusion follows from the first and the immediate fact that $d_Y^{1-\eps}$ is an ultrametric whenever $d_Y$ is an ultrametric.

Clearly, each map $\phi_i$ in $\phi = \{\phi_i\}_{i=1}^n$ is noncontractive as map from $(X,d_X^{1-\eps})$ into $(Y,d_Y^{1-\eps})$ since it is noncontractive as map from $(X,d_X)$ into $(Y,d_Y)$. Then we have
\begin{align*}
    \|\{\phi_i: (X,d_X^{1-\eps}) \to (Y,d_Y^{1-\eps})\}_{i=1}^n\|_{p}^p &= \sup_{x,y\in X}\int_0^\infty pt^{p-1}D_\phi(x,y)(t^{\frac{1}{1-\eps}})dt \\
    &= 1+\sup_{x,y\in X}\int_1^\infty pt^{p-1}D_\phi(x,y)(t^{\frac{1}{1-\eps}})dt \\
    &= 1+\sup_{x,y\in X}\int_1^\infty pt^{(1-\eps)(p-1)}D_\phi(x,y)(t)(1-\eps)t^{-\eps}dt \\
    &\leq 1+\int_1^\infty pt^{(1-\eps)(p-1)}t^{-p}D^p(1-\eps)t^{-\eps}dt \\
    &= 1+\int_1^\infty pt^{-1-\eps p}D^p(1-\eps)dt \\
    &= 1+D^p\frac{1-\eps}{\eps} < \frac{D^p}{\eps}.
\end{align*}
\end{proof}

The next lemma is implicit in multiple articles (e.g., \cite{Bartal,FRT}) but seems to evade explicit statement. We provide the complete proof for self-containment.

\begin{lemma} \label{lem:separationmodulusweak}
Suppose $(X,d)$ is a finite metric space with separation modulus $\sigma <\infty$. Then $X$ admits a weak-$L_1$-stochastic embedding into the class of ultrametric spaces with weak-$L_1$-stochastic distortion $\leq 64\sigma$.
\end{lemma}

\begin{proof}
We use the original construction of \cite{Bartal}. In the construction of $(\phi_i: X\to U_i)_i$, it will be the case that each target space $U_i$ is simply equal to $X$ as a set, but equipped with a different metric, and each map $\phi_i$ is the identity map.

Let $k_{min},k_{max} \in \Z$ such that, for all $x\neq y \in M$, $2^{k_{min}} < d(x,y) \leq 2^{k_{max}}$. For each $k \in [k_{min},k_{max}] \cap \Z$, choose a probability measure $\P^k$ on the set of $2^k$-bounded partitions of $(X,d)$, which we denote by $\PART_k$, such that, for each $x,y \in X$, $\P^k_P(P(x) \neq P(y)) \leq \sigma 2^{-k}d(x,y)$. Observe that $\PART_{k_{min}}$ contains only the singleton-set partition. We equip the product set $\Pi_{k=k_{min}}^{k_{max}} \PART_k$ with the product probability measure $\Pi_{k=k_{min}}^{k_{max}} \P^k$, which we abbreviate by $\P$.

There is a map $\Phi$ from $\Pi_{k=k_{min}}^{k_{max}} \PART_k$ to the set of ultrametrics $d_U$ on $X$ with $d_U(x,y) \geq d(x,y)$ for every $x,y \in X$. Let us denote this set by $\U_{\geq d}(X)$. Namely, we define
$$\Phi((P_k)_{k=k_{min}}^{k_{max}})(x,y) := \max\{2^{k+1}: P_k(x) \neq P_k(y)\}.$$
It is clear that $\Phi((P_k)_{k=k_{min}}^{k_{max}}) \in \U_{\geq d}(X)$. We think of the pushforward of the product probability measure $\Phi_\# \P$ on $\U_{\geq d}(X)$ to be the random noncontractive map from $X$ into the class of ultrametric spaces. It remains to calculate the weak-$L_1$-stochastic distortion of this random map.

Choose $\ell \in \N$ with $2^{\ell-1} \leq \sigma \leq 2^\ell$. Let $x\neq y \in X$, and choose $m \in [k_{min},k_{max}] \cap \Z$ with $2^{m-1} \leq d(x,y) \leq 2^{m}$. Let $t \geq 1$, and choose $n \in \N$ with $2^{n-1} \leq t \leq 2^n$. Then we have
\begin{align*}
    t (\Phi_\# \P)_{d_U} \left(d_U(x,y) \geq td(x,y) \right) &= t \P_{(P_k)_k} \left(\max\left\{2^{k+1}: P_k(x) \neq P_k(y)\right\} \geq td(x,y) \right) \\
    &\leq 2^n \P_{(P_k)_k} \left(\max\left\{2^{k+1}: P_k(x) \neq P_k(y)\right\} \geq 2^{n+m-2} \right) \\
    &= 2^n (1-\P_{(P_k)_k} \left(\max\left\{2^{k+1}: P_k(x) \neq P_k(y)\right\} < 2^{n+m-2} \right)) \\
    &= 2^n \left(1-\Pi_{k=n+m-3}^{k_{max}} \P^k_{P_k} (P_k(x) = P_k(y))\right) \\
    &\leq 2^n \left(1-\Pi_{k=n+m-3}^{k_{max}} \max\{0,1-2^{\ell+m-k}\}\right) \\
    &= \begin{cases}
        2^n & n \leq \ell+3 \\
        2^n \left(1-\Pi_{k=n+m-3}^{k_{max}}(1-2^{\ell+m-k})\right) & n > \ell+3
    \end{cases}  \\
    &\leq \begin{cases}
        2^n & n \leq \ell+3 \\
        2^{n+1} \sum_{k=n+m-3}^{\infty} 2^{\ell+m-k} & n > \ell+3
    \end{cases}  \\
    &= \begin{cases}
        2^n & n \leq \ell+3 \\
        2^{\ell+5} & n > \ell+3
    \end{cases}  \\
    &\leq 2^{\ell+5} \leq 2^6\sigma.
\end{align*}
\end{proof}

We now prove $(2) \implies (3)$ in Theorem~\ref{thm:mainextended}.

\begin{theorem} \label{thm:separation=>stochasticembedding}
Let $(X,d)$ be a metric space. If there exists a proper gauge $\omega$ such that $(X,\omega\circ d)$ has finite separation modulus, then there exists a proper gauge $\omega_0$ such that $(X,\omega_0 \circ d)$ admits an $L_1$-stochastic embedding into the class of ultrametric spaces.
\end{theorem}

\begin{proof}
Assume that there exists a proper gauge $\omega$ such that $(X,\omega\circ d)$ has separation modulus $\sigma<\infty$. Then by Lemma~\ref{lem:separationmodulusweak}, every finite subset of $(X,\omega\circ d)$ admits a weak-$L_1$-stochastic embedding into the class of ultrametric spaces with weak-$L_1$-stochastic distortion $\leq 64\sigma$. Then by Lemma~\ref{lem:weaksnowflake}, every finite subset of $(X,\sqrt{\omega}\circ d)$ admits an $L_1$-stochastic embedding into the class of ultrametric spaces with $L_1$-stochastic distortion $\leq 128\sigma$. By definition, this means that the whole space $(X,\sqrt{\omega}\circ d)$ admits an $L_1$-stochastic embedding into the class of ultrametric spaces. Since $\omega_0 := \sqrt{\omega}$ is a proper gauge, the proof is complete.
\end{proof}

\section{Non-Property A, SQI Small-Scale Expanders, and Poincar\'e Inequalities} \label{sec:nonPropA}
We will say that a graph $X$ is {\it non-Property A} if it fails to have Property A. In this section, we focus on non-Property A graphs and prove by contrapositive the implications $(5) \implies (6) \implies (7) \implies (1)$ from Theorem~\ref{thm:mainextended}.

\subsection{Non-Property A and sqi small-scale expanders}
\label{ss:nonPropA=>almostexpanders}
In this subsection we prove $(7) \implies (1)$ from Theorem~\ref{thm:mainextended} by contrapositive. That is, we show (in Theorem~\ref{thm:nonPropA=>almostexpander}) that non-Property A graphs contain a subquadratic isoperimetry small-scale expander sequence of induced subgraphs.

Let $H$ be a finite graph. Given a subset $F \sbs H$, let $\partial_H F$ denote the {\it internal vertex boundary} of $F$ with respect to $H$. That is, $x\in \partial_H F$ if and only if $x \in F$ and there exists $y \in H \setminus F$ with $x\sim y$. Note that any connected subset of $H$ having nonempty intersection with both $F$ and $H\setminus F$ must contain an element of $\partial_H F$.

Denote by $\Phi_H: \N \to [0,\infty]$ the {\it $\ell_1$-isoperimetric profile} of $H$. That is,
\begin{equation*}
    \Phi_H(v) := \min\left\{ \frac{|\partial_H A|}{|A|}:A\subseteq H,|A|\le v\right\}.
\end{equation*}
The $\ell_1$-isoperimetric profile is sometimes referred to as the conductance profile, see e.g., \cite{GoMonTe}. In graphs with degree bounded by $r$, the $\ell_1$-isoperimetric profile is always bounded by $r$. Also the value of $\Phi_{H}(v)$ can always be made small for $v > \frac{1}{2}|H|$, so it is only interesting to consider $v \leq \frac{1}{2}|H|$. Therefore, for a bounded degree sequence of finite graphs $(G_k)_k$ with $|G_k|\to\infty$ (meaning there exists $r\in\N$ such that the degree of each vertex in each graph $G_k$ is at most $r$), the strongest (qualitative) condition one can ask for is $\inf_n\inf_{v\leq\frac{1}{2}|G_k|} \Phi_{G_k}(v) \geq \eps_0 > 0$. Such a sequence is precisely a sequence of {\it expander graphs} (\cite[Definition~7.2.4]{Willett}, see also \cite{Lubotzky} for background). Any bounded degree graph containing a sequence of induced expander subgraphs $(G_k)_k$ is necessarily non-Property A (\cite[\S7.2]{Willett}), while the converse is not true, as the example of \cite{Z2homology} discussed in \S\ref{sec:intro} is a non-Property A graph that coarsely embeds into $L_1$ and thus contains no expanders (\cite[Proposition~7.2.5]{Willett}). However, in this section we are able to show that non-Property A graphs contain ``subquadratic isoperimetry small-scale expander sequences" (Theorem~\ref{thm:nonPropA=>almostexpander}), which we define below.

\begin{definition}[Subquadratic Isoperimetry Small-Scale Expanders] \label{def:almostexpander}
Let $(H_k)_k$ be a bounded degree sequence of finite graphs with $|H_k|\to\infty$. We say that $(H_k)_k$ is a {\it subquadratic isoperimetry (sqi) small-scale expander sequence} if there exist $\eps>0$ and $\alpha<1/2$ such that the $\ell_1$-isoperimetric profiles satisfy
\begin{align*}
\Phi_{H_k}(v) & \ge\eps \:\mbox{ for }\: v\le k,\\
\Phi_{H_k}(v) & > v^{-\alpha} \:\mbox{ for }\: k<v\le\tfrac{1}{2}|H_k|
\end{align*}
for all $k \in \N$.
\end{definition}

Since non-Property A graphs are non-locally hyperfinite by Theorem~\ref{thm:Elek}, it straightforward to construct induced subgraphs $(H_k)_k$ satisfying the first condition $\Phi_{H_k}(v) \ge\eps \:\mbox{ for }\: v\le k$ on their $\ell_1$-isoperimetric profile. Graphs like this are known as {\it non-uniformly locally amenable graphs}, and it was shown in \cite{Elek} that this condition is in fact equivalent to non-Property A. In order to obtain our main result of this section, Theorem~\ref{thm:nonPropA=>almostexpander}, we require some control on the profile $\Phi_{H_k}(v)$ for all values of $v$, even those above the scale $k$ of local expansion. It turns out that decay slower than $v \mapsto \sqrt{1/v}$ is the critical threshold needed to prove the implication $(6)\implies(7)$ from Theorem~\ref{thm:mainextended}, as will become apparent in the proof of Theorem~\ref{thm:almostexpander=>Poincare}, and hence this motivates our definition of sqi small-scale expander sequences. In the next several lemmas, we build up the necessary machinery in order to construct the sqi small-scale expander sequence of induced subgraphs of a non-Property A graph. We begin by recalling important quantitative notation and terminology.

Let $H$ be a finite graph, $\eps\in\R$, and $k\in\N$. Recall the definition of $(\eps,k)$-hyperfinite from Definition~\ref{def:hyperfinite}. If $H$ fails to be $(\eps,k)$-hyperfinite, we say that it is {\it non-$(\eps,k)$-hyperfinite}. It is convenient for our proofs in this subsection to allow $H = \emptyset$ and $\eps\leq 0$. Having said this, it is important to note that
\begin{itemize}
    \item if $\eps>0$ and $|H| \leq k$, then $H$ is necessarily $(\eps,k)$-hyperfinite, and
    \item if $\eps \leq 0$ and $H \neq \emptyset$, then $H$ is necessarily non-$(\eps,k)$-hyperfinite.
\end{itemize}
We will use these observations throughout the section.

An {\it $(\eps,k)$-F\o{}lner set} is a subset $F\sbs H$ with $|F|\leq k$ and $|\partial_H F| < \eps |F|$. We say that a finite graph $H$ is {\it $(\eps,k)$-amenable} if it contains an $(\eps,k)$-F\o{}lner set. If $H$ fails to be $(\eps,k)$-amenable, we say that is it {\it non-$(\eps,k)$-amenable}. The next lemma was essentially observed by Elek in \cite[Proposition 2.1]{Elek}. Since our definition of $(\eps,k)$-hyperfiniteness is slightly different than that used in \cite{Elek} (see Remark~\ref{rmk:equivalent}), we provide the proof for completeness.

\begin{lemma} \label{lem:amenable}
Let $H$ be a finite graph, $\eps\in\R$, and $k\in\N$. If $H$ is non-$(\eps,k)$-hyperfinite, then there exists an induced subgraph of $H$ that is simultaneously non-$(\eps,k)$-hyperfinite and non-$(\eps,k)$-amenable.
\end{lemma}

\begin{proof}
We prove the contrapositive by strong induction on $|H|$. The base case $|H|=1$ is trivial. Now let $H$ be a finite graph with $|H| \geq 2$, and assume that every induced subgraph of $H$ is $(\eps,k)$-hyperfinite or $(\eps,k)$-amenable. We will prove that $H$ itself is $(\eps,k)$-hyperfinite.

If $H$ is non-$(\eps,k)$-hyperfinite, then, by assumption, $H$ is $(\eps,k)$-amenable. Hence, we can find a subset $F \sbs H$ such that $|F| \leq k$ and $|\partial_H F| < \eps |F|$. Then the induced subgraph $H \setminus F$ is $(\eps,k)$-hyperfinite by the inductive hypothesis. Therefore, we can find subsets $\{B_i\}_i$ of $H \setminus F$ and a subset $Z \sbs H \setminus F$ such that
\begin{itemize}
    \item $\{B_i\}_i$ partitions $H\setminus F$,
    \item $|B_i| \leq k$ for all $i$,
    \item $|Z| < \eps |H\setminus F|$, and
    \item for every connected subset $C$ of $(H \setminus F) \setminus Z$, there is an $i$ such that $C \sbs B_i$.
\end{itemize}

Let $Y := \partial_H F \cup Z$. Then we have
\begin{itemize}
    \item $\{F\} \cup \{B_i\}_i$ partitions $H$,
    \item $|F|,|B_i| \leq k$ for all $i$,
    \item $|Y| < \eps|F| + \eps|H\setminus F| = \eps |H|$, and
    \item for every connected subset $C$ of $H \setminus Y$, we have that $C \sbs F$ or there is an $i$ such that $C \sbs B_i$.
\end{itemize}
This shows that $H$ is $(\eps,k)$-hyperfinite.
\end{proof}

The next lemma and its proof are similar to the previous one. One difference to note is that the constant $\eps$ in non-$(\eps,k)$-hyperfiniteness can change in Lemma~\ref{lem:nonhyperfinite}. This lemma is needed in the proof of the F\o{}lner Set Removal Lemma~\ref{lem:Folnerremoval} (as is Lemma~\ref{lem:amenable}).

\begin{lemma} \label{lem:nonhyperfinite}
Let $X$ be a finite graph, $H \sbs X$, $\eps\in\R$, and $k\in\N$. Suppose that $X$ is non-$(\eps,k)$-hyperfinite while the induced subgraph $X\setminus H$ is $(\eps,k)$-hyperfinite. Then the induced subgraph $H$ is non-$(\eps-\frac{|\partial_X H|}{|H|},k)$-hyperfinite.
\end{lemma}

\begin{proof}
Suppose towards a contradiction that the induced subgraph $H$ is $(\eps-\frac{|\partial_X H|}{|H|},k)$-hyperfinite. Let $\{B_i\}_i,Z$ be subsets of $H$ such that
\begin{itemize}
    \item $\{B_i\}_i$ partitions $H$,
    \item $|B_i| \leq k$ for all $i$,
    \item $|Z| < (\eps-\frac{|\partial_X H|}{|H|}) |H|$, and
    \item for every connected subset $C$ of $H \setminus Z$, there is an $i$ such that $C \sbs B_i$.
\end{itemize}
Let $\{B'_j\}_j,Z'$ be subsets of $X\setminus H$ such that
\begin{itemize}
    \item $\{B'_j\}_j$ partitions $X\setminus H$,
    \item $|B'_j| \leq k$ for all $j$,
    \item $|Z'| < \eps |X\setminus H|$, and
    \item for every connected subset $C$ of $(X\setminus H) \setminus Z'$, there is a $j$ such that $C \sbs B'_j$.
\end{itemize}

Set $\{A_m\}_m := \{B_i\}_i \cup \{B'_j\}_j$ and $Y:= Z \cup \partial_X H \cup Z'$. Then we have
\begin{itemize}
    \item $\{A_m\}_m$ partitions $X$,
    \item $|A_m| \leq k$ for all $m$,
    \item $|Y| < (\eps-\frac{|\partial_X H|}{|H|}) |H| + |\partial_X H| + \eps|X\setminus H| = \eps|X|$, and
    \item for every connected subset $C$ of $X \setminus Y$, there is an $m$ such that $C \sbs A_m$.
\end{itemize}
This shows that $X$ is $(\eps,k)$-hyperfinite, a contradiction.
\end{proof}

The following F\o{}lner Set Removal Lemma is our main technical tool of the subsection, and we will see that it quickly implies the existence of an sqi small-scale expander sequence of induced subgraphs in non-Property A graphs (Theorem~\ref{thm:nonPropA=>almostexpander}).

\begin{lemma}[F\o{}lner Set Removal] \label{lem:Folnerremoval}
Let $X$ be a finite graph. Let $\phi:\N\to[0,\infty)$ be a nonincreasing function such that $\sum_{n=0}^{\infty}\phi(2^n)<\infty$. If $\eps>0$, then there exists $k_0\in\N$ such that if $k\ge k_0$, then the following holds: if $X$ is non-$(\eps,k)$-hyperfinite, then there exists an induced subgraph $H_*\sbs X$ with $\ell_1$-isoperimetic profile satisfying
\begin{align*}
\Phi_{H_*}(v) & \ge\eps/2 \:\mbox{ for }\: v\le k,\\
\Phi_{H_*}(v) & > \phi(v) \:\mbox{ for }\: k<v\le\tfrac{1}{2}|H_*|.
\end{align*}
\end{lemma}

\begin{proof}
Assume that $\eps>0$. Define $\delta: \N \to [0,\infty)$ by
\begin{equation} \label{eq:delta}
    \delta(v) := \sum_{n=0}^\infty \phi(2^nv).
\end{equation}
Note that $\delta$ is finite-valued since $\phi$ is nonincreasing and $\sum_{n=0}^{\infty}\phi(2^n)<\infty$, and also $\delta$ is nonincreasing since $\phi$ is nonincreasing. Choose $k_0 \in \N$ such that $\delta(k_0) < \eps/2$. Such a choice of $k_0$ is possible because $\sum_{n=0}^{\infty}\phi(2^n)<\infty$. Assume that $k \geq k_0$, so that $\delta(k) \leq \delta(k_0) < \eps/2$.

Assume that $X$ is non-$(\eps,k)$-hyperfinite. Consider the collection $\H$ of all induced subgraphs $H \sbs X$ such that $|H| > k$ and $H$ is non-$(\eps-\delta(|H|),k)$-hyperfinite. Note that for any $H \in \H$, it holds that $\delta(|H|) \leq \delta(k) < \eps/2$. Thus,
\begin{equation} \label{eq:nonamenable}
    \eps - \delta(|H|) > \eps/2 > 0 \hspace{.2in} \forall H\in\H.
\end{equation}
Note that $\H \neq \emptyset$ since $X \in \H$ (here we use the facts that (1) $G$ being non-$(\eta,k)$-hyperfinite implies that $G$ is non-$(\eta',k)$-hyperfinite for all $\eta' \leq \eta$ and (2) $\eta>0$ and $G$ being non-$(\eta,k)$-hyperfinite implies that $|G| > k$, facts that we will continue to use throughout the proof). Let $H_* \in \H$ with
\begin{equation*}
    |H_*| = \min_{H\in\H} |H|.
\end{equation*}
We claim that the $\ell_1$-isoperimetric profile of $H_*$ satisfies the desired inequalities.

First, we observe that the induced subgraph $H_*$ must be non-$(\eps-\delta(|H_*|),k)$-amenable. Indeed, otherwise, Lemma~\ref{lem:amenable} would imply that there is a proper induced subgraph $H \subsetneq H_*$ which is non-$(\eps-\delta(|H_*|),k)$-hyperfinite. As observed in \eqref{eq:nonamenable}, $\eps-\delta(|H_*|) > 0$, and so the previous sentence implies $|H| > k$. Furthermore, since $\eps-\delta(|H_*|) \geq \eps-\delta(|H|)$, we have that $H$ is non-$(\eps-\delta(|H|),k)$-hyperfinite. Thus, $H\in\H$, contradicting the minimality of $|H_*|$. This proves that $H_*$ is non-$(\eps-\delta(|H_*|),k)$-amenable, and hence by \eqref{eq:nonamenable}, it is non-$(\eps/2,k)$-amenable, which in turn is easily seen to be equivalent to
\begin{align*}
    \Phi_{H_*}(v) & \ge\eps/2 \:\mbox{ for }\: v\le k.
\end{align*}
This verifies the first desired inequality for $\Phi_{H_*}$.

For the second desired inequality for $\Phi_{H_*}$, assume towards a contradiction that it fails. Then we can find $H \sbs H_*$ such that
\begin{align} \label{eq:Folner}
    k < |H| \leq \tfrac{1}{2}|H_*|, && \frac{|\partial_{H_*} H|}{|H|} \leq \phi(|H|).
\end{align}
We will show that $H\in\H$, which again will contradict the minimality of $|H_*|$ and complete the proof. We already have by \eqref{eq:Folner} that $|H|>k$, and so it remains to show that $H$ is non-$(\eps-\delta(|H|),k)$-hyperfinite.

If the induced subgraph $H_* \setminus H$ is non-$(\eps-\delta(|H_*|),k)$-hyperfinite, then since $\eps - \delta(|H_*|) > 0$ (by \eqref{eq:nonamenable}), it must happen that $|H_*\setminus H| > k$, and hence $H\setminus H_* \in \H$, contradicting minimality of $|H_*|$. Therefore, $H_* \setminus H$ must be $(\eps-\delta(|H_*|),k)$-hyperfinite. Then we can apply Lemma~\ref{lem:nonhyperfinite} to conclude that $H$ is non-$(\eps-\delta(|H_*|)-\frac{|\partial_{H_*}H|}{|H|},k)$-hyperfinite. Using \eqref{eq:Folner}, the definition of $\delta$, and monotonicity of $\phi$, it is clear that
$$\delta(|H_*|) + \frac{|\partial_{H_*}H|}{|H|} \leq \delta(2|H|) + \phi(|H|) \overset{\eqref{eq:delta}}{=} \delta(|H|),$$
and hence the previous sentence implies that $H$ is non-$(\eps-\delta(|H|),k)$-hyperfinite.
\end{proof}

We now arrive to the main result of this subsection. It implies the contrapositive of $(7)\implies(1)$ in Theorem~\ref{thm:mainextended}.

\begin{theorem}[SQI Small-Scale Expander Sequence in non-Property A Graph] \label{thm:nonPropA=>almostexpander}
Let $X$ be a bounded degree non-Property A graph. Then $X$ contain an sqi small-scale expander sequence of induced subgraphs.
\end{theorem}

\begin{proof}
We will show that the exponent $\alpha$ in the definition of sqi small-scale expander can be taken arbitrarily close to 0. Let $\alpha\in(0,\frac12)$. Since $X$ is non-Property A, by Theorem~\ref{thm:Elek}, there exists $\eps>0$ such that for every $k\in\N$, there exists a finite induced subgraph $H_k \sbs X$ that is non-$(\eps,k)$-hyperfinite. Then, by the F\o{}lner Set Removal Lemma~\ref{lem:Folnerremoval}, there exists $k_0\in\N$ such that for all $k\geq k_0$, there exists an induced subgraph $H'_k \sbs H_k$ with $\ell_1$-isoperimetic profile satisfying
\begin{align*}
\Phi_{H'_k}(v) & \ge\eps/2 \:\mbox{ for }\: v\le k,\\
\Phi_{H'_k}(v) & > v^{-\alpha} \:\mbox{ for }\: k<v\le\tfrac{1}{2}|H'_k|.
\end{align*}
Thus, $(H'_k)_{k=k_0}^\infty$ is an sqi small-scale expander sequence of induced subgraphs of $X$.
\end{proof}

\subsection{SQI small-scale expanders and Poincar\'e inequalities}
Let $H$ be a finite graph. When $f: H \to \R$ is a function and $\mu$ is a signed measure on $H$, we denote by $\langle f,\mu \rangle$ the integral $\int_H f d\mu$. Given a nonempty subset $P\subseteq H$, denote by $u_{P}$ the uniform probability measure on $P$. Take a random nonempty subset $\Pi \sbs H$ with law $\P$ on $\{0,1\}^{H} \setminus \{\emptyset\}$. We consider a Poincar\'e inequality of the form 
\begin{equation*}
\E_\Pi\left[\left|\left\langle f,u_{\Pi}-u_{H}\right\rangle \right|\right]\le\frac{C}{|H|}\sum_{x\sim y\in H}\left|f(x)-f(y)\right|,
\end{equation*}
where the sum is over all edges $\{x,y\}\sbs H$. 

\begin{definition}\label{def:Poincare}
For a random nonempty subset $\Pi \sbs H$, define the {\it Poincar\'e constant} of $\Pi$ as
$$C(\Pi,H):=\sup_{f\neq{\rm const.}}\frac{\E_\Pi\left[\left|\left\langle f,u_{\Pi}-u_{H}\right\rangle \right|\right]}{\left\Vert f\right\Vert _{W^{1,1}}},$$
where $\sup$ is taken over non-constant functions $f:H\to\R$, and 
$$\left\Vert f\right\Vert_{W^{1,1}}=\frac{1}{\left|H\right|}\sum_{x\sim y\in H}\left|f(x)-f(y)\right|$$
is the Sobolev seminorm of $f$.
\end{definition}

Next we estimate the Poincar\'e constants of Bernoulli nonempty point processes.

\begin{definition}[Bernoulli Nonempty Point Process] \label{def:Bernoullinonempty}
For $H$ a finite nonempty graph and $p\in(0,1)$, denote by $\P_p$ the product of Bernoulli$(p)$ measures on $\{0,1\}^{H}$, and by $\widetilde{\P}_{p}$ the conditional distribution
\begin{equation*}
    \widetilde{\P}_p(\Pi) :=
    \begin{cases}
    0 & \Pi = \emptyset \\
    \dfrac{\P_p(\Pi)}{1-\P_p(\emptyset)} & \Pi \in \{0,1\}^H\setminus\{\emptyset\}
    \end{cases} = \begin{cases}
    0 & \Pi = \emptyset \\
    \dfrac{p^{|\Pi|}(1-p)^{|H\setminus\Pi|}}{1-(1-p)^{|H|}} & \Pi \in \{0,1\}^H\setminus\{\emptyset\}
    \end{cases}.
\end{equation*}
Denote by $\widetilde{\Pi}(p)$ the random nonempty subset with law $\widetilde{\P}_p$. We call $\widetilde{\Pi}(p)$ the {\it Bernoulli$(p)$ nonempty point process} of $H$.
\end{definition}

\begin{example}[Cheeger Constant]
Let $\widetilde{\Pi}(0^+)$ denote $\lim_{p\to 0^+} \widetilde{\Pi}(p)$. Clearly, we have that $\P(\widetilde{\Pi}(0^+) = \{v\}) = 1/|H|$ for each $v\in H$. In other words, $\widetilde{\Pi}(0^+)$ is a uniformly random singleton subset of $H$. Its Poincar\'e constant $C(\widetilde{\Pi}(0^+),H)$ is
\begin{equation*}
    C(\widetilde{\Pi}(0^+),H) = \sup_{f\neq{\rm const.}}\frac{\E_{\widetilde{\Pi}(0^+)}\left[\left|\left\langle f,u_{\widetilde{\Pi}(0^+)}-u_{H}\right\rangle \right|\right]}{\left\Vert f\right\Vert _{W^{1,1}}} = \sup_{f\neq{\rm const.}}\frac{\sum_{v\in H}|f(v)-\langle f,u_H\rangle|}{\sum_{x\sim y\in H} |f(x)-f(y)|},
\end{equation*}
which is (comparable up to multiplicative constants by) the {\it Cheeger constant} of the graph $H$ \cite[Definition~1.2]{Lubotzky} (this can be seen by plugging in $f = \mathsf{1}_A$ into the ratio above for $A \sbs H$).
\end{example}

The above example shows that a bounded degree sequence of finite graphs $(H_n)_n$ with $|H_n|\to\infty$ is an expander sequence if and only if $\sup_n C(\widetilde{\Pi}(0^+),H_n) \leq C <\infty$. We will show in the main result of this subsection, Theorem~\ref{thm:almostexpander=>Poincare}, that an sqi small-scale expander sequence $(H_n)_n$ satisfies the similar but weaker condition $\sup_n C(\widetilde{\Pi}(p_n),H_n) \leq C <\infty$ for some $p_n\to 0$. Before proving this theorem, we need additional notation and a lemma which essentially yields the theorem when the nonconstant function $f: H \to \R$ in Definition~\ref{def:Poincare} is an indicator function $\mathsf{1}_A$. It is standard that it suffices to consider $f$ of this form (see, e.g., \cite[Theorem~2.3]{GO}).

Denote by $\Pi(p)$ the random nonempty subset with law $\P_p$. For $v\in H$, Let $\Pi_v(p)$ denote the random variable equal to 1 if $v\in\Pi(p)$ and equal to 0 if $v\not\in\Pi(p)$. Hence, $\{\Pi_v(p)\}_{v\in H}$ are iid Bernoulli$(p)$ random variables.

\begin{lemma}\label{lem:moments}
Let $H$ be a finite graph. Assume that $|H| \geq e^4$, and let $p \in [\frac{8\log |H|}{|H|},1]$. Then for all $A\sbs H$, we have
$$\E\left[\left|\left\langle {\bf 1}_{A},u_{\widetilde{\Pi}(p)}-u_{H}\right\rangle \right|\right]\le 9\min\left\{ \frac{p^{-1/2}|A|^{1/2}}{\left|H\right|},\frac{|A|}{|H|}\right\}.$$
\end{lemma}

\begin{proof}
Let $A \sbs H$. Obviously, we may assume that $A \neq \emptyset$. It is easy to see that the hypotheses $p \geq \frac{8\log|H|}{|H|} > \frac{32}{|H|}$ implies that $\widetilde{\P}_p(\Pi) \leq (1-e^{-32})^{-1}\P_p(\Pi) \leq \frac{9}{8}\P_p(\Pi)$ for all $\Pi\sbs \{0,1\}^V$. Hence, it suffices to prove
\begin{equation} \label{eq:bound}
    \E\left[\chi_{\{\Pi(p)\neq\emptyset\}}\left|\left\langle {\bf 1}_{A},u_{\Pi(p)}-u_H\right\rangle \right|\right] \leq 8\min\left\{ \frac{p^{-1/2}|A|^{1/2}}{\left|
    H\right|},\frac{|A|}{|H|}\right\}.
\end{equation}

For future use, let us note that a Chernoff bound yields
\begin{equation} \label{eq:Chernoff}
    \P(|\Pi(p)| < \tfrac{1}{2}p|H|) = \P\left(\sum_{v\in H}\Pi_v(p) < \tfrac{1}{2}p|H|\right) \leq e^{-\frac{1}{8}p|H|} \leq \frac{1}{|H|},
\end{equation}
where the last inequality follows from the hypothesis that $p \geq \frac{8\log|H|}{|H|}$. In particular, this implies
\begin{equation} \label{eq:Chernoffapp}
    \E\left[\frac{\chi_{\{\Pi(p)\neq\emptyset\}}}{|\Pi(p)|^2}\right]^{1/2} \leq \sqrt{e^{-\frac{1}{8}p|H|}} + \frac{2}{p|H|} \leq \frac{16e^{-1}+2}{p|H|} < \frac{8}{p|H|}.
\end{equation}

Returning to \eqref{eq:bound}, observe that the left-hand side is equal to
\begin{equation*}
    (\text{LHS}) = \E\left[\chi_{\{\Pi(p)\neq\emptyset\}}\left|\frac{1}{|\Pi(p)|}\sum_{v\in A}\Pi_v(p) - \frac{|A|}{|H|}\right|\right].
\end{equation*}

On the one hand, we can bound this by
\begin{align*}
    (\text{LHS}) &\leq \E\left[\frac{\chi_{\{\Pi(p)\neq\emptyset\}}}{|\Pi(p)|}\sum_{v\in A}\Pi_v(p)\right] + \frac{|A|}{|H|} \\
    &\overset{\eqref{eq:Chernoff}}{\leq} \frac{1}{|H|} + \frac{2}{p|H|}\E\left[\sum_{v\in A}\Pi_v(p)\right] + \frac{|A|}{|H|} \\
    &= \frac{1}{|H|} + \frac{2|A|}{|H|} + \frac{|A|}{|H|} \leq 4\frac{|A|}{|H|}.
\end{align*}
This establishes one of the desired inequalities in \eqref{eq:bound}.

On the other hand, we can also use the bound 
\begin{align*}
    (\text{LHS}) &= \E\left[\left|\frac{\chi_{\{\Pi(p)\neq\emptyset\}}}{|\Pi(p)|}\left(\sum_{v\in A}\Pi_v(p) - \frac{|A|}{|H|}|\Pi(p)|\right)\right|\right] \\
    &\leq \E\left[\frac{\chi_{\{\Pi(p)\neq\emptyset\}}}{|\Pi(p)|^2}\right]^{1/2} \E\left[\left(\sum_{v\in A}\Pi_v(p) - \frac{|A|}{|H|}|\Pi(p)|\right)^2\right]^{1/2} \\
    &\overset{\eqref{eq:Chernoffapp}}{\leq} \frac{8}{p|H|} \E\left[\left(\sum_{v\in A}\Pi_v(p) - \frac{|A|}{|H|}|\Pi(p)|\right)^2\right]^{1/2} \\
    &= \frac{8}{p|H|} \E\left[\left(\left(1-\frac{|A|}{|H|}\right)\sum_{v\in A}\Pi_v(p) - \frac{|A|}{|H|}\sum_{v\in H\setminus A}\Pi_v(p)\right)^2\right]^{1/2}.
\end{align*}
The term inside the square root is of the form $\E[(X-Y)^2]$ where $X$ and $Y$ are independent variables with $\E[X] = \E[Y]$. Hence, this term is equal to $\Var(X) + \Var(Y)$. Therefore we have
\begin{align*}
    \frac{1}{8}(\text{LHS}) &\leq \frac{1}{p|H|}\E\left[\left(\left(1-\frac{|A|}{|H|}\right)\sum_{v\in A}\Pi_v(p) - \frac{|A|}{|H|}\sum_{v\in H\setminus A}\Pi_v(p)\right)^2\right]^{1/2} \\
    &= \frac{1}{p|H|}\left( \left(1-\frac{|A|}{|H|}\right)^2 \Var\left(\sum_{v\in A}\Pi_v(p)\right) + \left(\frac{|A|}{|H|}\right)^2 \Var\left(\sum_{v\in H\setminus A}\Pi_v(p)\right) \right)^{1/2} \\
    &= \frac{1}{p|H|}\left( \left(1-\frac{|A|}{|H|}\right)^2 |A|p(1-p) + \left(\frac{|A|}{|H|}\right)^2 |H\setminus A|p(1-p) \right)^{1/2} \\
    &\leq \frac{p^{-1/2}}{|H|}\left( \left(1-\frac{|A|}{|H|}\right)^2 |A| + \left(\frac{|A|}{|H|}\right)^2 |H\setminus A| \right)^{1/2} \\
    &= \frac{p^{-1/2}|A|^{1/2}}{|H|}\left(1-\frac{|A|}{|H|}\right)^{1/2} \leq \frac{p^{-1/2}|A|^{1/2}}{|H|}.
\end{align*}
This establishes the other desired inequality in \eqref{eq:bound}.
\end{proof}

We now bound the Poincar\'e constant $C(\widetilde{\Pi}(p_k),H_k)$ for an sqi small-scale expander sequence $(H_k)_k$. This will follow quickly from Lemma~\ref{lem:moments} and the previously mentioned fact that in the definition of $C(\Pi,H)$, one can take the nonconstant function $f: H \to \R$ to be an indicator. Theorem~\ref{thm:almostexpander=>Poincare} yields the contrapositive of $(6)\implies(7)$ in Theorem~\ref{thm:mainextended}.

\begin{theorem} \label{thm:almostexpander=>Poincare}
Let $(H_k)_k$ be an sqi small-scale expander sequence of graphs. Then there exist a constant $C<\infty$ and a sequence of probabilities $(p_k)_k$ with $p_k\to 0$ such that the Poincar\'e constants satisfy $\sup_k C(\widetilde{\Pi}(p_k),H_k) \leq C$.
\end{theorem}

\begin{proof}
Let $\eps>0$ and $\alpha\in(0,\frac{1}{2})$ such that the $\ell_1$-isoperimetric profiles satisfy
\begin{align*}
\Phi_{H_k}(v) & \ge\eps \:\mbox{ for }\: v\le k,\\
\Phi_{H_k}(v) & > v^{-\alpha} \:\mbox{ for }\: k<v\le\tfrac{1}{2}|H_k|.
\end{align*}
Since $|H_k|\to\infty$, we may disregard an initial segment and assume that $|H_k| \geq e^4$ for all $k$. We will show that $C := 9\max\{1,1/\eps\}$ and $p_k := \max\{k^{2\alpha-1},\frac{8\log |H_k|}{|H_k|}\}$ works. By \cite[Theorem~2.3]{GO}, it suffices to prove that
\begin{equation} \label{eq:Poincare}
\E\left[\left|\left\langle \mathsf{1}_A,u_{\widetilde{\Pi}(p_k)}-u_{H_k}\right\rangle \right|\right] \le \frac{C}{|H_k|}\sum_{x\sim y\in H_k}\left|\mathsf{1}_A(x)-\mathsf{1}_A(y)\right| = C\frac{|\partial_{H_k}A|}{|H_k|}
\end{equation}
for every $k$ and every $A \sbs H_k$ with $|A| \leq \frac{1}{2}|H_k|$.

Let $A \sbs H_k$ with $|A| \leq \frac{1}{2}|H_k|$. First assume that $|A|\le k$. Then by Lemma~\ref{lem:moments} and $\Phi_{H_k}(v)\ge\eps$ for $v\le k$, we have 
$$\E\left[\left|\left\langle {\bf 1}_{A},u_{\widetilde{\Pi}(p_k)}-u_{H_k}\right\rangle \right|\right]\le 9\frac{|A|}{|H_k|}\le\frac{9}{\eps}\frac{|\partial_{H_k}A|}{|H_k|}.$$
This establishes inequality \eqref{eq:Poincare} in this case.

The other case is $|A|>k$. since $p_k\ge k^{2\alpha-1}$ and $\alpha < \frac12$, we have that $p_k^{-1/2}|A|^{1/2}<|A|^{1-\alpha}$, and thus Lemma~\ref{lem:moments} and $\Phi_{H_k}(v)>v^{-\alpha}$ for $k < v \leq \frac{1}{2}|H_k|$ yields
\begin{align*}
\E\left[\left|\left\langle  \mathsf{1}_{A},u_{\widetilde{\Pi}(p_k)}-u_{H_k}\right\rangle \right|\right] & \le \frac{9p_k^{-1/2}|A|^{1/2}}{|H_k|}\le \frac{9|A|^{1-\alpha}}{|H_k|} \le \frac{9|\partial_{H_k}A|}{|H_k|}.
\end{align*}
This establishes \eqref{eq:Poincare} and completes the proof.
\end{proof}

\subsection{Poincar\'e inequalities and $L_1$-distortion}
In this subsection, we will see how an upper bound on the Poincar\'e constant $C(\Pi,H)$ and a lower bound $\E[\|u_{\Pi}-u_H\|_{\LF}]$ results in a lower bound on the $L_1$-distortion of $\LF(H)$ and $W_1(H)$. Poincar\'e inequalities have found great utility in the past to prove $L_1$-distortion lower bounds of various classes of discrete metric spaces (see, for example, \cite[Chapter~4]{Ostrovskiibook} and \cite{NY}). Their application to lower bound the $L_1$-distortion of Lipschitz free spaces only recently appeared in \cite{GO}.

We begin by lower bounding $\E[\|u_{\widetilde{\Pi}(p)}-u_H\|_{\LF}]$ for $\widetilde{\Pi}(p)$ the Bernoulli$(p)$ nonempty point process of a semimetric space $H$. In order to do so, we need two lemmas.

\begin{lemma} \label{lem:sparse1}
Let $(H,d)$ be a finite, nonempty semimetric space and $\Pi \sbs H$ a (deterministic) nonempty subset. Let $R<\infty$ be such that $\max_{x\in H}|B_R(x)| < \frac{|H|}{2|\Pi|}$, where $B_R(x) = \{y\in H: d(x,y)<R\}$. Then there exists $Y \sbs H$ such that $|Y| > \frac{1}{2}|H|$ and $\dist(\Pi,Y) \geq R$.
\end{lemma}

\begin{proof}
 Set $Y := \{y\in H: \dist(\Pi,y) \geq R\}$. We will show that $|H\setminus Y| < \frac{1}{2}|H|$, from which the conclusion will follow. Note that $H \setminus Y$ can be covered by $\{B_{R}(x)\}_{x\in \Pi}$. Using this, together with the definition of $R$, we get
$$|H\setminus Y| \leq \sum_{x\in \Pi}|B_{R}(x)| < |\Pi|\tfrac{|H|}{2|\Pi|} = \tfrac{1}{2}|H|.$$
\end{proof}

\begin{lemma} \label{lem:sparse2} 
Let $H,\Pi,R$ be as in the statement of Lemma~\ref{lem:sparse1}. Then
$$\|u_\Pi-u_H\|_{\LF} \geq \tfrac{1}{2}R.$$ 
\end{lemma}

\begin{proof}
By Lemma~\ref{lem:sparse1}, we can find $Y \sbs H$ such that $|Y| > \frac{1}{2}|H|$ and $\dist(\Pi,Y) \geq R$. Let $f: H \to \R$ be the $1$-Lipschitz function given by $f(y) = \dist(\Pi,y)$ for all $y\in H$. Then $f\geq0$, $f\big|_{\Pi}=0$, and $f\big|_{Y}\geq R$. Therefore, we have
$$\|u_\Pi-u_H\|_{\LF} \geq \int f du_H - \int fdu_\Pi \geq \int_Y f du_H \geq Ru_H(Y) \geq \tfrac{1}{2}R.$$ 
\end{proof}

We can now obtain our desired lower bound on $\E[\|u_{\widetilde{\Pi}(p)}-u_H\|_{\LF}]$.

\begin{proposition} \label{prop:LF(Bernoulli-u)}
Let $H$ be a finite, nonempty semimetric space with $|H|\geq e^4$ and $p\in [\frac{8\log|H|}{|H|},1]$. Let $\widetilde{\Pi}(p)$ be the Bernoulli$(p)$ nonempty point process of $H$, and let $R<\infty$ be such that $\max_{x\in H}|B_R(x)| < \frac{1}{4p}$. Then
$$\E\|u_{\widetilde{\Pi}(p)}-u_H\|_{\LF} \geq \tfrac{1}{3}R.$$
\end{proposition}

\begin{proof}
It is easy to see that the hypotheses $p \geq \frac{8\log|H|}{|H|} > \frac{32}{|H|}$ implies that $\widetilde{\P}_p(\Pi) \leq (1-e^{-32})^{-1}\P_p(\Pi)$ for all $\Pi\sbs \{0,1\}^H$. Using this together with a Chernoff bound yields
\begin{align*}
    \P(\widetilde{\Pi}(p) > 2p|H|) &\leq (1-e^{-32})^{-1} \P(\Pi(p) > 2p|H|) \\
    &\leq (1-e^{-32})^{-1} e^{-\frac{1}{3}p|H|} \\
    &\leq (1-e^{-32})^{-1} e^{-\frac{32}{3}} < \tfrac13.
\end{align*}
Hence, $\P\left(\frac{|H|}{2|\widetilde{\Pi}(p)|} > \tfrac{1}{4p}\right) > \tfrac23$. Thus, by Lemma~\ref{lem:sparse2}, we get 

\begin{align*}
    \E\|u_{\widetilde{\Pi}(p)}-u_H\|_{\LF} &\geq (\tfrac23)\tfrac{1}{2}R = \tfrac13 R.
\end{align*}
\end{proof}

We now arrive at the main result of this subsection, which proves $(5)\implies(6)$ from Theorem~\ref{thm:mainextended} by contrapositive.

\begin{proposition} \label{prop:c1(LF(almostexpanders))}
Let $(H_k)_k$ be a bounded degree sequence of graphs and $(d_k)_k$ a sequence of semimetrics on $H_k$ such that
\begin{enumerate}
    \item $|H_k|\to\infty$,
    \item there exists a sequence of probabilities $(p_k)_k$ such that $p_k\to 0$ and $\sup_k C(\widetilde{\Pi}(p_k),H_k) < \infty$,
    \item there exists $L<\infty$ such that $\sup_k \max_{x\sim y \in H_k}d_k(x,y) \leq L$, and
    \item for all $R<\infty$, we have $\sup_{k}\max_{x\in H_k}|\{y\in H_k: d_k(x,y)< R\}| <\infty$.
\end{enumerate}
Then the linear $\{\ell_1^m\}_{m=1}^\infty$-distortion of $\LF(H_k,d_k)$ tends to $\infty$ as $k\to\infty$.
\end{proposition}

\begin{proof}
Let $(m_k)_k$ be a sequence in $\N$ and $(T_k: \LF(H_k,d_k) \to \ell_1^{m_k})_k$ a sequence of noncontractive linear maps, with coordinate functionals $T_k = (T_k^i)_{i=1}^{m_k}$. We will show that the operator norm $\|T_k\|$ tends to $\infty$ as $k\to\infty$, which will prove the proposition.

For each $k$, let $R_k<\infty$ be such that $\max_{x\in H_k}|\{y\in H_k: d_k(x,y) < R_k\}| < \frac{1}{4p_k}$. Since $p_k\to 0$, hypothesis (4) implies that $R_k$ can be chosen to tend to $\infty$, and we make this choice. Let $r\in\N$ such that for all $k$, each vertex of $H_k$ has degree at most $r$.

For each $k$ and $i \leq m_k$, let $f_k^i: H_k \to \R$ be a function representing the linear coordinate functional $T_k^i: \LF(H_k,d_k) \to \R$ (see \S\ref{ss:LF}). By hypothesis (2), we have that there exists $C<\infty$ such that
\begin{equation} \label{eq:Poincare2}
    \E|\langle f_k^i, u_{\widetilde{\Pi}(p_k)}-u_{H_k}\rangle| \leq \frac{C}{|H_k|}\sum_{x\sim y\in H_k} |f_k^i(y)-f_k^i(x)|
\end{equation}
for all $k$ and all $i$. Then for all sufficiently large $k$, we have that
\begin{align*}
    \tfrac{1}{3}R_k &\overset{\text{Prop }\ref{prop:LF(Bernoulli-u)}}{\leq} \E\|u_{\widetilde{\Pi}(p_k)}-u_{H_k}\|_{\LF} \\
    &\leq \E\|T_k(u_{\widetilde{\Pi}(p_k)}-u_{H_k})\|_1 \\
    &= \E\sum_i|T_k^i(u_{\widetilde{\Pi}(p_k)}-u_{H_k})| \\
    &= \sum_i\E|T_k^i(u_{\widetilde{\Pi}(p_k)}-u_{H_k})| \\
    &= \sum_i\E|\langle f_k^i, u_{\widetilde{\Pi}(p_k)}-u_{H_k}\rangle| \\
    &\overset{\eqref{eq:Poincare2}}{\leq} \sum_i \frac{C}{|H_k|}\sum_{x\sim y\in H_k} |f_k^i(y)-f_k^i(x)| \\
    &= \frac{C}{|H_k|}\sum_{x\sim y\in H_k} \sum_i|f_k^i(y)-f_k^i(x)| \\
    &= \frac{C}{|H_k|}\sum_{x\sim y\in H_k} \sum_i|T_k^i(\delta_y-\delta_x)| \\
    &= \frac{C}{|H_k|}\sum_{x\sim y\in H_k}\|T_k(\delta_y-\delta_x)\|_1 \\
    &\leq \frac{C}{|H_k|}\sum_{x\sim y\in H_k}\|T_n\|\|\delta_y-\delta_x\|_{\LF} \\
    &= \frac{C}{|H_k|}\sum_{x\sim y\in H_k}\|T_k\|d_k(x,y) \\
    &\leq \frac{CL}{|H_k|}\sum_{x\sim y\in H_k}\|T_k\|L \\
    &\leq \frac{CLr}{2}\|T_k\|.
\end{align*}
This inequality shows that $\|T_k\| \gtrsim R_k \to \infty$ as $k\to\infty$.
\end{proof}

With the help of Lemma~\ref{lem:linearreduction}, we can upgrade the previous proposition from nonembeddability of Lipschitz free spaces to nonembeddability of Wasserstein-1 spaces.

\begin{theorem} \label{thm:Poincare=>L1distortion}
Let $X$ be a connected, bounded degree graph equipped with the shortest path metric. If $X$ contains a sequence of induced subgraphs $(H_k)_k$ such that $|H_k|\to\infty$ and $\sup_k C(\widetilde{\Pi}(p_k),H_k) < \infty$ for some sequence of probabilities $(p_k)_k$ with $p_k\to 0$, then for any metric space $Y$ coarsely containing $X$, the Wasserstein space $W_1(Y)$ fails to biLipschitzly embed into $L_1$.
\end{theorem}

\begin{proof}
Assume that $X$ contains a sequence of induced subgraphs $(H_k)_k$ such that $|H_k|\to\infty$ and $\sup_k C(\widetilde{\Pi}(p_k),H_k) < \infty$ for some sequence of probabilities $(p_k)_k$ with $p_k\to 0$. Let $f: X \to Y$ be a coarse embedding into a metric space $(Y,d_Y)$. Consider the pullback semimetric $d_f$ on $X$ given by $d_f(x,y) := d_Y(f(x),f(y))$. Since $X$ is a bounded degree graph and $f$ is a coarse embedding, it can be seen that $d_f$ satisfies the properties
\begin{enumerate}
    \item there exists $C<\infty$ such that $\sup_{x\sim y \in X}d_f(x,y) \leq C$ and
    \item for all $R<\infty$, we have $\sup_{x\in X}|\{y\in X: d_f(x,y)< R\}| <\infty$.
\end{enumerate}
Note that, by construction, $f: (X,d_f) \to Y$ is an isometric embedding.

Let $d_k$ denote the restriction of $d_f$ to $H_k$. Then the above two properties of $d_f$ imply that the hypotheses of Proposition~\ref{prop:c1(LF(almostexpanders))} are satisfied for $(d_k)_k$, and hence the linear $\{\ell_1^m\}_{m=1}^\infty$-distortion of $\LF(H_k,d_k)$ tends to $\infty$. Since the linear $\{\ell_1^m\}_{m=1}^\infty$-distortion of $\LF(H_k,d_k)$ is the same as the biLipschitz $L_1$-distortion of $W_1(H_k,d_k)$ (Lemma~\ref{lem:linearreduction}) and since each $(H_k,d_k)$ isometrically embeds into $Y$, we get that $W_1(Y)$ does not biLipschitzly embed into $L_1$.
\end{proof}

\bibliographystyle{alpha}

\end{document}